\documentclass[review]{elsarticle}

\usepackage{lineno,hyperref}

\usepackage[utf8]{inputenc}

\usepackage{amsmath,amssymb,amsthm,amsfonts}
\usepackage[overload]{empheq}
\usepackage{bbm}
\usepackage{cases}
\usepackage{dsfont}

\newcommand{\II}{\mathds{1}}
\newcommand{\JJ}{\mathbb{J}}
\newcommand{\ii}{\mathbbm{i}}
\newcommand{\jj}{\mathbbm{j}}
\newcommand{\kk}{\mathbbm{k}}
\newcommand{\smalleftarrow}{{\mkern-2mu\leftarrow}}
\newcommand{\smalleftarrowscript}{{\mkern-2mu\leftarrow}}

\newcommand{\vertiii}[1]{{\left\vert\kern-0.25ex\left\vert\kern-0.25ex\left\vert #1 
    \right\vert\kern-0.25ex\right\vert\kern-0.25ex\right\vert}}

\DeclareMathOperator{\E}{\mathbb{E}}
\def\P{\mathbb{P}}

\newtheorem{thm}{Theorem}
\newtheorem{lemma}{Lemma}
\newtheorem{rem}{Remark}
\theoremstyle{plain}
\newcounter{ass}
\newenvironment{ass}%
{\begin{enumerate}
\stepcounter{ass}
\renewcommand{\labelenumi}{\textbf{\theenumi}}
\renewcommand{\theenumi}{(A.\arabic{ass})}
\item }
{\end{enumerate}}%

\begin{document}

\begin{frontmatter}

\title{Hydrodynamic limit of a stochastic model of proliferating cells with chemotaxis.%
}

\author{Rados{\l}aw Wieczorek}
\address{University of Silesia, ul. Bankowa 14, 40-007 Katowice, Poland.}
\ead{radoslaw.wieczorek@us.edu.pl}

\begin{abstract}
A hybrid stochastic individual-based model of proliferating 
cells with chemotaxis is presented. The model is expressed 
by a branching diffusion process coupled to a partial differential equation 
describing concentration of chemotactic factor.
It is shown that in the hydrodynamic limit when number of cells goes to infinity
the model converges to the solution of a nonconservative Patlak-Keller-Segel-type system.
A nonlinear mean-field stochastic model is defined and it is proven that the 
movement of descendants of a single cell in the individual model 
converges to this mean-field process.
\end{abstract}

\end{frontmatter}

\section{Introduction}

This paper is concerned with a stochastic model of biological cells that undergo chemotaxis and  proliferate.
The hydrodynamic limit of this model is shown to solve the equation of Patlak-Keller-Segel type with
cell proliferation of the form
\begin{subequations}
  \label{eqn:PKS}
  \begin{align}[left ={\empheqlbrace}]
   \label{eqn:PKSp}
 	\partial _t p(t,\mathbf{x})=&\frac{1}{2}\Delta p(t,\mathbf{x})+\nabla\left(p(t,\mathbf{x})\mathbf{b}(\mathbf{x},\nabla\varrho)\right)+\lambda(\mathbf{x}, \nabla\varrho(t,\mathbf{x}))p(t,\mathbf{x})\\
   \label{eqn:PKSrho}
 	\partial _t \varrho(t,\mathbf{x})=&D\Delta\varrho(t,\mathbf{x})- r \varrho(t,\mathbf{x})+ \alpha [\kappa\! * p(t,\cdot)] (\mathbf{x}) .
  \end{align}
\end{subequations}

Mathematical description of biological cells undergoing chemotaxis, 
i.e. moving in response to gradients of chemical factors, has a long tradition and 
still attracts great interest. The seminal papers by 
Patlak \cite{Patlak1953} and Keller and Segel \cite{KellerSegel1970} 
stimulated the whole branch of applied mathematics
(one can find extensive reviews  \cite{Horstmann2003,HillenPainter2009} with many citations).
The limit passages between microscopic, kinetic models of collective cell behaviour 
to the macroscopic description with PDE's
has received great attention in recent four decades
\cite{Oelschlager84,Sznitman1984,Morale-et-al-05,RudnickiWieczorek2006a,FournierJourdain2017,Wieczorek2015,BudhirajaFan2017,%
CapassoWieczorek2020,BoiCapassoMorale2000}.
The first rigorous proof of the convergence of a stochastic particle system 
to chemotaxis equations system was given in \cite{Stevens2000} following the work of \cite{Oelschlager1989}.
Later, many authors continued these ideas \cite{HwangKangStevens2005,FournierJourdain2017,BudhirajaFan2017,BubbaLorenziMacfarlane2020}.	
Macroscopic mean-field limits of weakly interacting particles 
where investigated also in \cite{Meleard1996,BolleyGuillinVillani2007,KotelenezKurtz2010,LiLiuYu2019}.
Most particle approximations of PKS type equations consider cells and chemical particles in a similar way
\cite{Stevens2000,HwangKangStevens2005,FournierJourdain2017}, 
while they have completely different scale.
Si it is natural to consider cells as stochastic particles and chemical factor as continuous field
described by a partial differential equation \cite{CapassoMorale2013,BudhirajaFan2017,CapassoFlandoli2018,CapassoWieczorek2020}.
Although usually the equation or system describing the cell population is conservative,
i.e. it preserves total mass (the number of particles),
in many models it is natural to assume that the cell population is not constant and
the cells proliferate 
\cite{WoodwardTysonMyerscoughMurrayBudreneBerg1995,Wang2000,WangFordHarvey2008,BansayeSimatos2015}.
It is crucial for angiogenesis models \cite{CapassoMoraleFacchetti2012,McDougallWatsonDevlinMitchellChaplain2012,CapassoMorale2013,CapassoWieczorek2020},
where the proliferation is responsible for vessel branching.

In this paper we use the pathwise propagation of chaos approach similar to that of \cite{BudhirajaFan2017},
but we allow for the proliferation of cells. 
Therefore, the number of cells (or total mass) is not conserved 
and, moreover, the individual process has noncontinuous trajectories.
To retain the pathwise description and convergence, we write the birth and death process 
as a solution of stochastic equations (cf. \cite{GarciaKurtz2006}).


We construct a sequence of processes indexed by the initial number of particles $n_0$.
The description of our process can be divided into three components:
\begin{description}
\item{\textbf{Movement of cells.}}
The cells move according to the following SDE
\begin{equation}
\label{eqn:particlemovement}
d\mathbf{X}^{n_0}_{i,\jj}(t) =   
	\mathbf{b}(\mathbf{X}^{n_0}_{i,\jj}(t), \nabla\varrho_{n_0}(t,\mathbf{X}^{n_0}_{i,\jj}(t)))dt +  
	\sigma \,dW_{i,\jj}(t),
\end{equation}
where $W_{i,\jj}(t)$ are independent Brownian motions, $\sigma$ is the diffusion coefficient
and $\mathbf{b}$ is the (chemotactic) drift that depends on the position  of  a cell and 
the gradient of the concentration $\varrho_{n_0}$ of some chemical factor.
One can also allow for the dependence of $\mathbf{b}$ on the concentration $\varrho_{n_0}$ itself,
not only on its gradient, and all facts and proofs of the paper remain true,
but for the sake of shortness and simplicity of the notation we 
neglect this dependence.  
The indexing $i,\jj$ will be explained later. 

\item{\textbf{Equation for chemoreactant.}}
The concentration $\varrho_{n_0}$ of chemotactic factor satisfies the following PDE
\begin{equation}
\label{eqn:Drho_n}
\frac{\partial \varrho_{n_0}(t,\mathbf{x})}{\partial t} = 
D \Delta \varrho_{n_0}(t,\mathbf{x}) -r \varrho_{n_0}(t,\mathbf{x}) + \alpha\, \kappa\! *\! \xi^n_t\,(\mathbf{x}),
\end{equation}
where $D$, $r$ and $\alpha$  are diffusion, degradation and production rates. 
The measure
\[
\xi^{n_0}_t = \frac{1}{n_0}\sum_{{i,\jj}}\delta_{X^{n_0}_{i,\jj}(t)},
\]
is the empirical measure of all cells
and function $\kappa$ is a mollifying kernel that represents the fact that cells are actually not points, 
but have spatial size. The spatial convolution
\begin{align*}
\kappa\! * \xi^{n_0}_t\,(\mathbf{x}) 
	=& \int_{\mathbb{R}^d} \kappa(x-y) \xi^{n_0}_t(dy) = 
\frac{1}{n_0} \sum_{i,\jj}\int_{\mathbb{R}^d} \kappa(x-y)\delta_{X_{i,\jj}(t)}(dy)
\\=&
\frac{1}{n_0} \sum_{i,\jj} \kappa(x-X_{i,\jj}(t))
\end{align*}
is a mollified version of the empirical measure describing spatial positions of cells,
responsible for the production of the chemoreactant.
 From the mathematical point of view this allows us to consider 
classical solutions to equation \eqref{eqn:Drho_n}.

\item{\textbf{Cell population dynamics.}}
We assume that cells may die or proliferate with rates depending on the chemoreactant. 
Death means that a cell disappears and 
proliferation means that a cell dies leaving two new daughter cells at the same place as the mother cell.
The birth rate of a cell placed at $x$ at time $t$ 
depend on the position and on the the concentration of chemoreactant 
and is given by $\lambda_{\text{b}}(\mathbf{x},\varrho^{n_0}(t,\mathbf{x}))$
and the mortality rate is $\lambda_{\text{d}}(\mathbf{x},\varrho^{n_0}(t,\mathbf{x}))$.
\end{description}

\vspace{1em}
The main goal of the paper is to prove that in hydrodynamic limit the model converge
to the solutions of \eqref{eqn:PKS} (Theorem \ref{thm:convinD}), 
and, most importantly, to prove that, if the initial number of cells tends to infinity,
the trajectories of the descendants of a single cell converge to the trajectories
of the hybrid nonlinear mean-field model defined in section \ref{sec:hybrid} (Theorem \ref{thm:conv}).
To this aim, we also want to rigorously define the described individual processes
and show wellposedness. 

We call the model hybrid for three reasons: first is that in the base model the discrete individual-based model is coupled to the
continuous description of chemoreactant by PDEs%
\cite{McDougallWatsonDevlinMitchellChaplain2012,CapassoMorale2013,CapassoWieczorek2020}. 
Second, the stochastic trajectories of cells have jumps due to birth and death of cells \cite{BujorianuLygeros2006,BuckwarRiedler2011}.
Third, pertains specifically to the intermediate model described in section 2.2, which is obtained by the convergence 
of fully stochastic individual based models in Theorem \ref{thm:conv}.
In this hybrid model, the stochastic particle system is coupled to the deterministic solution of the asymptotic PDE \eqref{eqn:PKS}.

The article is organised as follows. 
In the next section  we introduce the notation and present 
the rigorous definitions of the microscopic model and two `hybrid' mean field models
and write up the macroscopic equations. Section \ref{sec:wellp} is devoted to the presentation of assumptions and
results concerning the wellposedness. 
Section \ref{sec:conv} contains the convergence results.
The proofs are presented in section \ref{sec:proofs}.

\section{Definitions of the processes}
\label{sec:def}
In this section we formally define the considered processes. We start with the individual, fully microscopic model.
\subsection{Definition of the microscopic processes}
\label{sec:IBM}%
Note that if we use the empirical measure approach 
even for a simple two Brownian particles case, then giving the initial condition $\delta_{X_1^0}+\delta_{X_2^0}$
and two Brownian motions $W_1$ and $W_2$ does not  guarantee pathwise uniqueness 
--- we need to know the order of particles. 
The problem gets harder if the number of particles varies in time.
Since our goal is to obtain a pathwise propagation of chaos result, 
we need to define a process in a more direct way. 

Therefore, we construct the process in the following way. 
We mark particles by means of a subtree of Ulam-Harris tree, namely let
\begin{equation}
\JJ = \bigcup_{n\in\mathbb{N}_0}{\{0,1\}}^n
\end{equation}
 with a convention that ${\{0,1\}}^0=\emptyset$ means the root. 
Elements of $\JJ$ will be written as blackboard bold lowercase letters such as $\ii,\jj,\kk$.
If we write $\kk=\jj0$, we mean that $\kk$ is longer by one then $\jj$ 
and is created from $\jj$ by adding $0$ at the end. 
Moreover, we will denote by
$\jj^\smalleftarrow$ element of $\JJ$ obtained from $\jj$ by removing last number,
eg. if $\jj=\emptyset101$, then $\jj^\smalleftarrow=\emptyset10$.
Moreover, let us assume that
\begin{ass}\label{ass:WN}
$\left(W_{i,\jj}\right)_{i\in\mathbb{N},\jj\in\JJ}$ is an infinite array 
of independent $d$-dimensional standard Wiener processes, 
and $\left(\mathcal{N}_{i,\jj}\right)_{i\in\mathbb{N},\jj\in\JJ}$  is an array
of independent standard (i.e. such that intensity is Lebesgue measure) Poisson point processes on $[0,\infty)\times[0,\infty)$
on a complete probability space $(\Omega,\mathcal{F},\P)$.
\end{ass}
These processes will be used as a source of randomness in all processes, 
in particular in Eq. \eqref{eqn:IBMXn} and \eqref{eqn:hybridX}.

We actually define a sequence of processes indexed by the initial number of cell $n_0$.
Let us assume that we start with $n_0$ cells and let each initial 
cell be described by its position in $\mathbb{R}^d$,
the number $i$, and $\jj=\emptyset$ denoting that it is the first cell in its own tree of inheritance.
The initial cells are located at $\mathbf{X}_{i,\emptyset}(0)\in \mathbb{R}^d$, $i=1,\dots,n_0$.
At each branching event, two new cells appear as daughters of a cell described by $\mathbf{X}^{n_0}_{i,\jj}$
at the same place as the mother cell.
The daughter cells inherit the position, the cell line number $i$ and obtain new subsequent indices 
$\kk_1=\jj0$ and $\kk_2=\jj1$.
So, each cell is described by a triple 
$(\mathbf{x},i,\jj)\in{\mathbb{R}^d}\times \{1,\dots,n_0\}\times \JJ$
and possesses its own Brownian motion $W_{i,\jj}$ and 
a Poisson clock $\mathcal{N}_{i,\jj}$.
Let $\tau^{n_0}_{i,\jj}$ denote the moment when $i,\jj$-th cell appears 
--- that is $\tau^{n_0}_{(i,\emptyset)}$ is always 0,
and let $\sigma^{n_0}_{i,\jj}$ be the moment when $i,\jj$-th cell dies 
(the production of daughter cells also means death of the mother).
The time of death  $\sigma^{n_0}_{i,\jj}$ of an $(i,j)$-th cell is defined as a minimal $\sigma$ for which
\begin{multline}
\label{eqn:sigmadef}
\mathcal{N}_{i,\jj}\biggl(
	(t,z):\,z\in	
		\left[
				0,
				\lambda_{\text{b}}\bigl(\mathbf{X}^{n_0}_{i,\jj}(t),\varrho^{n_0}(t,\mathbf{X}^{n_0}_{i,\jj}(t))\bigr)
				+\lambda_{\text{d}}\bigl(\mathbf{X}^{n_0}_{i,\jj}(t),\varrho^{n_0}(t,\mathbf{X}^{n_0}_{i,\jj}(t))\bigr)
		\right],\\
		t\in\left[\tau^{n_0}_{i,\jj},\sigma\right]
\biggr)
=1
\end{multline}
with a convention that $\min \emptyset =\infty$.
The times of birth  $\tau^{n_0}_{(i,\jj0)}=\tau^{n_0}_{(i,\jj1)}$ of daughters of $i,\jj$-th cell 
are defined as a minimal $\tau$ for which
\begin{equation}
\label{eqn:taudef}
\mathcal{N}_{i,\jj}\left(
	(t,z):\,z\in	
		\left[
				0,\lambda_{\text{b}}\bigl(\mathbf{X}^{n_0}_{i,\jj}(t),\varrho^{n_0}(t,\mathbf{X}^{n_0}_{i,\jj}(t))\bigr)
		\right],
		t\in\left[\tau^{n_0}_{(i, \jj)},\tau \right]
\right)
=1.
\end{equation}
Clearly, not for every $\jj\in\JJ$ a cell will exist, 
e.g. if $\kk$-th cell dies, no cell with index 
created from $\kk$ by appending any zeros or ones cannot be born. 
In this case we have $\tau^{n_0}_{i,\jj}=\sigma^{n_0}_{i,\jj}=\infty$.

The movement of a $i,\jj$-th cell between time $\tau^{n_0}_{i,\jj}$ and  $\sigma^{n_0}_{i,\jj}$ 
is given by
\begin{subequations}
\label{eqn:IBM}
\begin{equation}
\label{eqn:IBMXn}
d\mathbf{X}^{n_0}_{i,\jj}(t) =  \mathbf{b}(\mathbf{X}^{n_0}_{i,\jj}(t), \nabla\varrho^{n_0}(t,\mathbf{X}^{n_0}_{i,\jj}(t)))dt 
											 + \sigma \,dW_{i,\jj}(t).
\end{equation}
with initial condition
$\mathbf{X}^{n_0}_{i,\jj}\left(\sigma^{n_0}_{i,\jj}\right)=
	\mathbf{X}^{n_0}_{i,\jj^\smalleftarrowscript}\left(\sigma^{n_0}_{i,\jj}\right)$ 
coupled with the equation for nutrient field
\begin{equation}
\label{eqn:IBMrhon}
\frac{\partial \varrho^{n_0}(t,\mathbf{x})}{\partial t} = 
D \Delta \varrho^{n_0}(t,\mathbf{x}) -r \varrho^{n_0}(t,\mathbf{x})  + \alpha[K * \xi^{n_0}_t](\mathbf{x})
\end{equation}
\end{subequations}
with $\varrho^{n_0}(0,\cdot)=\varrho_0$,
where $\xi^{n_0}_t$, given by 
\begin{equation}
\label{eqn:xi_n}
\xi^{n_0}_t = \frac{1}{n_0}\sum_{{i,\jj}}
			\II_{\left[\sigma^{n_0}_{i,\jj},\tau^{n_0}_{i,\jj}\right)}(t)\delta_{X^{n_0}_{i,\jj}(t)},
\end{equation}
is an empirical measure of all cells alive at time $t$.

Note that the generating processes $W_{i,\jj}$ and $\mathcal{N}_{i,\jj}$
have intentionally no index $n_0$. In order to obtain pathwise convergence, 
they are shared by processes with all $n_0$. They are defined for all $i\in\mathbb{N}$, 
but the definition of $n_0$-th process uses only those with $i=1,\dots,n_0$.
\begin{rem}
Note, that equation \eqref{eqn:IBMXn} can be written in a form
\begin{align*}
\mathbf{X}^{n_0}_{i,\jj}(t) =& \mathbf{X}^{n_0}_{i,\jj}\left(\sigma^{n_0}_{i,\jj}\right)
	+ \int_{\sigma^{n_0}_{i,\jj}}^t \mathbf{b}(\mathbf{X}^{n_0}_{i,\jj}(s), \nabla\varrho^{n_0}(s,\mathbf{X}^{n_0}_{i,\jj}(s)))\,ds 
&&\\&	+ \sigma \left(W_{i,\jj}(t)-W_{i,\jj}(\sigma^{n_0}_{i,\jj})\right),
	&&\text{for }
	t\in \left[\sigma^{n_0}_{i,\jj},\tau^{n_0}_{i,\jj}\right).
\end{align*}
It does not demand using Ito integral and has a pathwise unique solution.
\end{rem}

\subsection{Hybrid mean-field model}
\label{sec:hybrid}
The next model considered will be the limit of the individual-based model (cf. Theorems \ref{thm:convinD} and \ref{thm:conv}).
We consider one initial cell at position $\mathbf{\bar {X}_\emptyset}(0)=\mathbf{\bar {X}_{1,\emptyset}}(0)$ 
with the same population dynamics as before. Its descendants will be denoted by $\mathbf{\bar {X}_\jj}$ with
$\jj\in\JJ$ and their birth and death times are $\bar\sigma_{\jj}$ and $\bar\tau_{\jj}$, respectively,
defined in analogous way to \eqref{eqn:sigmadef}-\eqref{eqn:taudef} 
with $\varrho^{n_0}$ replaced by $\varrho$ and the same Poisson clocks
$\mathcal{N}_{1,\jj}$, as for the first cell line in each microscopic model.
The movement of $\jj$-th cell during its life is given by
\begin{subequations}
\label{eqn:hybrid}
\begin{equation}
\label{eqn:hybridX}
d\mathbf{\bar {X}_\jj}(t) =  
	F\bigl(\mathbf{\bar X}_\jj, \nabla \varrho(t,\mathbf{\bar X}_\jj)\bigr)dt + \sigma \,dW_{1,\jj}(t),
\end{equation}
coupled with the mean-field chemoreactant equation
\begin{equation}
\label{eqn:hybridrho}
\dfrac{\partial \varrho(t,\mathbf{x})}{\partial t} = 
D \Delta \varrho(t,\mathbf{x}) -r \varrho(t,\mathbf{x})  + \alpha(K\! * \bar\mu_t)(\mathbf{x}),
\end{equation}
\end{subequations}
with $\varrho(0,\cdot)=\varrho_0$,
where $\bar\mu_t$ is the mean of the empirical measure of $\left(\mathbf{\bar {X}_\jj}(t)\right)_{\jj\in\JJ}$,
namely
\begin{equation}
\label{eqn:barmu}
\bar\mu_t(A)=\E\bar\xi_t(A)
=\E\left[
\sum_{{i,\jj}}
			\II_{\left[\bar\tau_{\jj},\bar\sigma_{\jj}\right)}(t)\delta_{X^{n_0}_{i,\jj}(t)}(A)
	\right],
\text{ for } A\in\mathcal{B}(\mathbb{R}^d),
\end{equation}
where $\bar\xi_t$, $\bar\tau_{\jj}$ and $\bar\sigma_{\jj}$ are defined analogously to 
$\xi^{n_0}_t$, $\tau^{n_0}_{i,\jj}$ and $\sigma^{n_0}	_{i,\jj}$.
\begin{rem}
Note that \eqref{eqn:hybridrho} differs from \eqref{eqn:IBMrhon} 
only by replacing $\xi^{n_0}_t$ by $\bar\mu_t$.
\end{rem}

\subsection{Macroscopic model: Patlak-Keller-Segel type equation with proliferation}
The limit macroscopic model is given by the system of equations
\begin{subequations}
  \label{eqn:macro}
  \begin{align}[left ={\empheqlbrace}]
   \label{eqn:macrop}
 	\partial _t p(t,\mathbf{x})=&\frac{1}{2}\Delta p(t,\mathbf{x})+\nabla\left(p(t,\mathbf{x})\mathbf{b}(\mathbf{x},\nabla\varrho)\right)+\lambda(\mathbf{x}, \nabla\varrho(t,\mathbf{x}))p(t,\mathbf{x})\\
   \label{eqn:macrorho}
 	\partial _t \varrho(t,\mathbf{x})=&D\Delta\varrho(t,\mathbf{x})- r \varrho(t,\mathbf{x})+ \alpha [\kappa\! * p(t,\cdot)] (\mathbf{x}) 
  \end{align}
\end{subequations}
with $\lambda=\lambda_{\text{b}}-\lambda_{\text{d}}$ where
[$\kappa\! * p(t,\cdot)](\mathbf{x}) = \int_{\mathbb{R}^d} \kappa(x-y) p(t,y)dy$.
\subsection{Second hybrid model}
One of the motivations of the convergence result of this paper is to show the possibility
of replacing in simulations a multiparticle model \eqref{eqn:sigmadef}-\eqref{eqn:IBMrhon}
by a one with smaller number cells. However, the hybrid model \eqref{eqn:hybridX}-\eqref{eqn:hybridrho}
still includes the number of proliferating cells.
Therefore, we present here alternative, much simpler version of the hybrid model
that is related to the limit of the inidivilual-based one.
Namely, let us consider a single cell which which moves according to the same rule as the particles in previous models,
and endow it with a variable $M(t)$ denoting its mass:
\begin{subequations}
  \label{eqn:XMrho}
  \begin{align}[left ={\empheqlbrace}]
  \label{eqn:XMrhoX}
	&d\mathbf{X}(t) =  \mathbf{b}(\mathbf{X}(t),\nabla\varrho(t,\mathbf{X}(t)))dt + \sigma \,dW(t),\\
  \label{eqn:XMrhoM}
	&dM(t) = \lambda \bigl(\mathbf{X}(t),\varrho(t,\mathbf{X}(t))\bigr)M(t),\\
  \label{eqn:XMrhoR}
	&\dfrac{\partial \varrho(t,\mathbf{x})}{\partial t} = 
		D \Delta \varrho(t,\mathbf{x}) -r \varrho(t,\mathbf{x}) + \alpha(K * \mu_t)(\mathbf{x}),
  \end{align}
\end{subequations}
where 
\begin{equation}
\label{eqn:def-mu}
\mu_t(A)=\E\left[M(t)\II_{A}(X(t))\right]
=\int_{A\times [0,\infty)} m\, \P_{(\mathbf{X},M)}(d\mathbf{x},dm) \text{ for } A\in\mathcal{B}(\mathbb{R}^d)
\end{equation}
is the \textit{average  mass in the area $A$}.
\begin{rem}
The measure $\mu_t$ is equal to $\bar\mu_t$ given by \eqref{eqn:barmu}.
Note that equations \eqref{eqn:hybridrho} and \eqref{eqn:XMrhoR}
are then the same.
These facts will be proven and used in the proof of wellposedness of the hybrid model 
\end{rem}
\begin{rem}
Note moreover, that if $\mu_0$ is absolutely continuous, then the density of $\mu_t$ 
(and therefore $\bar\mu_t$) satisfies \eqref{eqn:macrop}.
In that case all equations \eqref{eqn:hybridrho}, \eqref{eqn:macrorho} and \eqref{eqn:XMrhoR}
coincide.
\end{rem}

\subsection{Other remarks}
Usually, the aim of a rigorous proof of convergence of individual models to macroscopic ones is twofold:
from one hand side, it is convenient to be able to derive  the macroscopic model of population
from the primitive rules that govern the individuals
and to know if and why the macroscopic model 
properly approximates the collective behaviour.
On the other hand side, 
often the individual based models are easier to simulate and the proofs of convergence
play significant role in justifying the use of 
individual-based model simulations as Monte Carlo methods for complicated PDEs.
We have also the third motivation that justifies the hybrid mean-field model.
Sometimes (cf. \cite{CapassoWieczorek2020,CapassoFlandoli2018,BonillaCapassoAlvaroCarretero2014})
simulating the individual-based model is costly, as in our case where it demands simulating 
branching diffusion coupled to PDE.
On the other hand side, solving only the macroscopic PDE can be insufficient, because we do not follow the geometry of trajectories.
In such a case we can use the hybrid mean-field model in the following way:
firstly solve numerically the non-stochastic  macroscopic model \ref{eqn:macro} 
for density of cells and concentration of chemoreactant,
and then simulate the branching diffusion \ref{eqn:hybridX+bd}
with already given evolution of chemoreactant $\varrho$.

\section{Assumptions and wellposedness}
\label{sec:wellp}
One of main goals of this paper is to make the definition of the individual model as strict as possible 
while keeping it readable. To that aim, besides the description in  section \ref{sec:IBM}, 
we need to define the state space of the process, which can be done in various ways.
We add to the space of positions $\mathbb{R}^d$ an additional state $\phi$ denoting a nonexisting cell
and we describe the state of all particles as an infinite array of points from $\mathbb{R}^d\cup\{\phi\}$
indexed by $(i,\jj)\in\mathbb{N}\times \JJ$ such that only finite numbers of elements are different then $\phi$,
that is 
\begin{equation}
\label{eqn:def-XX}
\mathbb{X} = \left\{
\left(\mathbf{x}_{i,\jj}\right)_{(i,\jj)\in\mathbb{N}\times \JJ}:\mathbf{x}_{i,\jj}\in\mathbb{R}^d\cup\{\phi\}, 
\text{ such that } \#\{\mathbf{x}_{i,\jj}:\mathbf{x}_{i,\jj}\neq\phi\} \text{ is finite}
\right\}
\end{equation}
with a natural metrics
\begin{equation}
d_{\mathbb{X}}(\mathbbm{x},\mathbbm{y})
= \max_{(i,\jj)\in\mathbb{N}\times \JJ} |\mathbf{x}_{i,\jj}-\mathbf{y}_{i,\jj}|,\quad 
\text{ for } \mathbbm{x}=\left(\mathbf{x}_{i,\jj}\right)_{(i,\jj)\in\mathbb{N}\times \JJ},
\mathbbm{y}=\left(\mathbf{y}_{i,\jj}\right)_{(i,\jj)\in\mathbb{N}\times \JJ}\in\mathbb{X},
\end{equation}
with a convention that $|\mathbf{x}-\phi|=1$ for any $\mathbf{x}\in\mathbb{R}^d$.
Because of the birth and death process the trajectories of the microscopic 
process are not continuous. We apply the standard convention to use c\'adl\'ag paths,
so the space of trajectories will be the Skorokhod space $D_{\mathbb{X}}[0,\infty)$.

Now, we can formally describe the branching diffusion component of the microscopic model 
as the solution to the following system of SDEs
\begin{align}
\notag%
\mathbf{X}^{n_0}_{i,\jj}\left(t\right)
=&\mathbf{X}^{n_0}_{i,\jj}\left(0\right)
	+ \int_0^t \mathbf{b}(\mathbf{X}^{n_0}_{i,\jj}(s), \nabla\varrho(t,\mathbf{X}^{n_0}_{i,\jj}(s)))\,ds
	+ \sigma\int_0^t \II_{\mathbb{R}^d}\left(\mathbf{X}^{n_0}_{i,\jj}(s)\right)\, dW_{i,\jj}(s),\\
&+
\int_0^t \chi^{n_0}_{\rm b}(\mathbf{X}^{n_0}_{i,\jj^\smalleftarrow}(s^-),z) \mathcal{N}_{i,\jj^\smalleftarrow}(ds,dz)
+
\int_0^t \chi^{n_0}_{\rm d}(\mathbf{X}^{n_0}_{i,\jj}(s^-),z) \mathcal{N}_{i,\jj}(ds,dz)
\end{align}
for $i\in\{1,\dots,n_0\}$, $\jj\in\JJ$
with
\begin{align*}
\chi^{n_0}_{\rm b}(\mathbf{x},z)=&
\begin{cases}
-\phi+\mathbf{x}, 
\text{ if }\mathbf{x}\neq\phi,\;
z \le \lambda_{\text{b}}(\mathbf{x},\varrho^{n_0}(\mathbf{x})),
\\
0,\text{ otherwise, }
\end{cases}
\\
\chi^{n_0}_{\rm d}(\mathbf{x},z)=&
\begin{cases}
\phi, \text{ if }\mathbf{x}\neq\phi,\;
z \le \lambda_{\text{b}}(\mathbf{x},\varrho^{n_0}(\mathbf{x}))
		+\lambda_{\text{d}}(\mathbf{x},\varrho^{n_0}(\mathbf{x})),
\\
0,\text{ otherwise, }
\end{cases}
\end{align*}
where we use a convention that $\phi-\phi=0\in\mathbb{R}^d$ 
and $\mathbf{x}+\phi=\phi$ 
for any $\mathbf{x}\in\mathbb{R}^d$
and $\mathbf{b}(\phi,x)=0$ for $x\in\mathbb{R}$.

In a similar way, to fully describe the branching diffusion of the hybrid model 
we add birth and death events to the equation \eqref{eqn:hybridX} obtaining
\begin{align}
\label{eqn:hybridX+bd}
\mathbf{\bar X}_{\jj}\left(t\right)
=&\mathbf{\bar X}_{\jj}\left(0\right)
	+ \int_0^t \mathbf{b}(\mathbf{\bar X}_{\jj}(s), \nabla\varrho(t,\mathbf{\bar X}_{\jj}(s)))\,ds
	+ \sigma\int_0^t \II_{\mathbb{R}^d}\left(\mathbf{\bar X}_{\jj}(s)\right)\, dW_{1,\jj}(s),
\\
\notag%
&+
\int_0^t \chi_{\rm b}(\mathbf{\bar X}_{\jj^\smalleftarrow}(s^-),z) \mathcal{N}_{1,\jj^\smalleftarrow}(ds,dz)
+
\int_0^t \chi_{\rm d}(\mathbf{\bar X}_{\jj}(s^-),z) \mathcal{N}_{1,\jj}(ds,dz),
\end{align}
for $\jj\in\JJ$, where $\chi_{\rm b}$ and $\chi_{\rm d}$ are defined like $\chi^{n_0}_{\rm b}$ and $\chi^{n_0}_{\rm d}$
with $\varrho ^{n_0}$ replaced by $\varrho $.

The second component of the microscopic process is the solution to the PDE \eqref{eqn:IBMrhon}.
We use here a semigroup notation, so the solution 
will be denoted as  $\varrho^{n_0}_t(\mathbf{x})$ instead of $\varrho^{n_0}(t,\mathbf{x})$, where
$\rho^{n_0}_t\in C^2_b(\mathbb{R}^d)$, 
and the space of trajectories will be $C_{C^2_b(\mathbb{R}^d)}{[0,\infty)}$ --- 
the space of continuous functions from $[0,\infty)$
to the space of bounded twice continuously differentiable functions with bounded second derivatives $C^2_b(\mathbb{R}^d)$.

Let $S_t$ be a semigroup generated by the operator $Af(\mathbf{x})=D\Delta f(\mathbf{x}) - r\,f(\mathbf{x})$, $f\in C^2_b(\mathbb{R}^d)$. 
Then we can rewrite equation \eqref{eqn:IBMrhon} as
\begin{equation}
\label{eqn:rhon0trewritten}
\varrho^{n_0}_t(\mathbf{x}) 
	= S_t \varrho_0(\mathbf{x}) 
	+ \alpha\int_0^t S_{t-s}[K * \xi^{n_0}_s](\mathbf{x})\,ds
\end{equation}
and, analogously, equation \eqref{eqn:hybridrho} as
\begin{equation}
\label{eqn:rhotrewritten}
\varrho_t(\mathbf{x}) 
	= S_t \varrho_0(\mathbf{x}) 
	+ \alpha\int_0^t S_{t-s}[K * \bar\mu_s](\mathbf{x})\,ds,
\end{equation}
and similarly \eqref{eqn:XMrhoR} with $\bar\mu_s$ replaced by $\mu_s$.

We will use the following assumptions:
\begin{ass}
\label{ass:const}
let $\sigma$, $D$, $r$, and $\alpha$ be positive constants.
\end{ass}
\begin{ass}
\label{ass:Fl}
let $\lambda_{\text{b}},\lambda_{\text{d}}\in C^1_b\left(\mathbb{R}^d\times \mathbb{R}_+\right)$
be nonnegative functions and  
$\lambda_{\text{b}}+\lambda_{\text{d}}<\bar\lambda$ for some constant $\bar\lambda>0 $;
let $\mathbf{b}:\mathbb{R}^d\times \mathbb{R}\to\mathbb{R}$ be bounded and boundedly differentiable;
let $\kappa\in C^2_b\left(\mathbb{R}^d\right)$ such that
$\int_{\mathbb{R}^d} \kappa(\mathbf{x})d\mathbf{x}=1$
and $\Delta \kappa$ is Lipschitz and let $L_{\kappa}$ be maximum of Lipschitz coefficients for $\kappa$ and $\Delta \kappa$.
\end{ass}

Now we can state the well-posedness theorems. Their proofs will be given in section \ref{sec:proofs-wp}.
\begin{thm}
\label{thm:well-p-IBM}%
Suppose that $\varrho_0\in C^2_b(\mathbb{R}^d)$ and 
$\mathbf{X}^{n_0}_{i,\emptyset}(0)$ for $i=1,\dots,n_0$
are independent random variables with probability law $\mu_0\in\mathcal{P}(\mathbb{R}^d)$.
Let assumptions \ref{ass:WN}-\ref{ass:Fl} be satisfied. 
For any $T>0$ there exists a process 
$(\mathbbm{x}^{n_0},\varrho^{n_0})=
\left(\bigl(\mathbf{X}^{n_0}_{i,\jj}\bigr)_{i\in\{1,\dots,n_0\},\jj\in\JJ},\varrho^{n_0}\right)
\in D_{\mathbb{X}}{[0,T]}\times C_{C^2_b(\mathbb{R}^d)}{[0,T]}$
described by \eqref{eqn:IBMXn} and \eqref{eqn:IBMrhon} with population dynamics 
given by \eqref{eqn:sigmadef} and \eqref{eqn:taudef} 
and it is pathwise uniquely defined.
\end{thm}

\begin{thm}
\label{thm:well-p-hybrid}%
Suppose that $\varrho_0\in C^2_b(\mathbb{R}^d)$ and 
$\bar{\mathbf{X}}_{\emptyset}(0)$ is a random variable with probability law $\mu_0\in\mathcal{P}(\mathbb{R}^d)$.
Let assumptions \ref{ass:WN}-\ref{ass:Fl} be satisfied. 
For any $T>0$ there exists a unique function $\varrho\in C_{C^2_b(\mathbb{R}^d)}{[0,T]}$
and the hybrid mean-field process 
$\mathbbm{\bar x}=\bigl(\mathbf{\bar X}_{\jj}\bigr)_{\jj\in\JJ}$ with trajectories in $D_{\mathbb{X}}[0,T]$
described by \eqref{eqn:hybrid} with \eqref{eqn:hybridX+bd} and it is pathwise uniquely defined.
\end{thm}
For completeness we state also the existence-uniqueness theorem for the macroscopic model. 
We do not present its proof, which is straightforward thanks to the regularization by $\kappa$. 
It goes by simple fixed point argument.
\begin{thm}
\label{thm:well-p-macro}%
Let assumptions \ref{ass:WN}-\ref{ass:Fl} be satisfied. 
If $\varrho_0$ and $p_0$ are in $C^2_b(\mathbb{R}^d)$, then
there exists a unique classical solution to the system \eqref{eqn:macro}.
\end{thm}
Now, we have the wellposedness theorem for the second mean-field model.
\begin{thm}
\label{thm:well-p-XM}
Let assumptions \ref{ass:const}-\ref{ass:Fl} be satisfied and $W$ be 
a $d$-dimensional standard Wiener process. 
Suppose that $\varrho(0,\cdot)=\varrho_0\in C^2_b(\mathbb{R}^d)$,
${\mathbf{X}}(0)$ is a random variable with probability law $\mu_0\in\mathcal{P}(\mathbb{R}^d)$
and $M(0)=1$.
Then system \eqref{eqn:XMrho}-\eqref{eqn:def-mu} has a unique solution.
\end{thm}

\section{Convergence results}
\label{sec:conv}
The approach to the convergence of particle systems is based on the propagation of chaos results
from \cite{BudhirajaFan2017} and \cite{Sznitman1991,CattiauxGuillinMalrieu2008},
namely, 
the proof of Theorem \ref{thm:conv}  is based on the fact that processes for all $n_0$ and the limit are coupled by use of the same 
probability space and the same processes 
$\left(W_{i,\jj}\right)_{i\in\mathbb{N},\jj\in\JJ}$ and $\left(\mathcal{N}_{i,\jj}\right)_{i\in\mathbb{N},\jj\in\JJ}$.
Since the processes have not continuous trajectories it demands more delicate approach.
We will use the following notation:
\begin{itemize}
\item denote by
$\mathbbm{x}^{n_0}(t)={\bigl(\mathbf{X}^{n_0}_{i,\jj}(t)\bigr)_{i\in\{1,\dots,n_0\},\jj\in\JJ}}$, 
coupled with $\varrho^{n_0}_t$, 
$t\in[0,T]$, 
the solution of microscopic model defined by \eqref{eqn:sigmadef}-\eqref{eqn:xi_n}.
\item let 
$\mathbbm{x}^{n_0}_1(t)={\bigl(\mathbf{\bar X}^{n_0}_{1,\jj}(t)\bigr)_{\jj\in\JJ}}$, $t\in[0,T]$ 
denote the branch of process $\mathbbm{x}^{n_0}$ starting from the first cell $\mathbf{\bar X}^{n_0}_{1,\jj}(0)$
driven by processes 
$\left(W_{1,\jj}\right)_{\jj\in\JJ}$ and $\left(\mathcal{N}_{1,\jj}\right)_{\jj\in\JJ}$.
\item now 
we define a mean field processes 
$\bar{\mathbbm{x}}(t)={\bigl(\mathbf{\bar X}_{\jj}(t)\bigr)_{\jj\in\JJ}}$ and  
$\varrho_t$, $t\in[0,T]$, 
defined by 
\eqref{eqn:hybridX+bd}-\eqref{eqn:hybridrho}
driven by processes 
$\left(W_{1,\jj}\right)_{\jj\in\JJ}$ and $\left(\mathcal{N}_{1,\jj}\right)_{\jj\in\JJ}$.
\item $\xi^{n_0}_t$ is a process given by \eqref{eqn:xi_n} which can be written as
\begin{equation*}
\xi^{n_0}_t = \frac{1}{n_0}\sum_{i=1}^{n_0}\sum_{\jj\in\JJ}
			\II_{\mathbb{R}^d}(X^{n_0}_{i,\jj}(t))\delta_{X^{n_0}_{i,\jj}(t)}.
\end{equation*}
\end{itemize}
Now we can state two convergence theorems.
\begin{thm}
\label{thm:convinD}
Let assumptions \ref{ass:WN}-\ref{ass:Fl} be satisfied. 
The for each $T>0$
\begin{enumerate}
\renewcommand{\labelenumi}{\textbf{\theenumi}}
\renewcommand{\theenumi}{(\roman{enumi})}
\item the sequence of processes $\xi^{n_0}_t$ converges in distribution
to $\bar\mu_t$ defined by \eqref{eqn:barmu} on $D_{\mathcal{M}}[0,T]$ with Skorokhod topology. 
The space $\mathcal{M}$ is considered here with a topology of vague convergence.
\item the sequence of processes $\varrho^{n_0}_t$ converges to
$\varrho_t$ given by \eqref{eqn:rhotrewritten} in distribution on $C_{C_b(\mathbb{R}^d)}[0,T]$.
\end{enumerate}
\end{thm}
Our next goal it to obtain pathwise convergence of $\mathbbm{x}^{n_0}_1$ to $\mathbbm{\bar x}$. 
We use the approach of \cite{BudhirajaFan2017}, but since both processes 
$\mathbbm{x}^{n_0}_1$ and $\mathbbm{\bar x}$ have discontinuous trajectories
the result is not as strong as there. 
The reason is, that jumps can occur not only in different moments of time, 
but also in different directions,
and after such a jump the processes are irrevocably not more close to each other.
Nevertheless, we have the following fact.
\begin{thm}
\label{thm:conv}
Let assumptions \ref{ass:WN}-\ref{ass:Fl} be satisfied. 
Then for each $T>0$
\begin{equation}
\sup_{t\in[0,T]}d_{\mathbb{X}}\left(
\mathbbm{x}^{n_0}_1(t),\mathbbm{\bar x}(t)
\right)
\end{equation}
converges to $0$ in probability.
\end{thm}


\section{Proofs}\label{sec:proofs}
\subsection{Wellposedness}\label{sec:proofs-wp}
In this section we prove firstly Theorem \ref{thm:well-p-IBM}, 
then Theorem \ref{thm:well-p-XM} and eventually Theorem \ref{thm:well-p-hybrid}, 
since its proof depends on fragments of two previous proofs.

We will repeatedly use the following fact
\begin{lemma}
\label{lm:Nt}
For any $T>0$ and all $n_0\in \mathbb{N}$
\[
\E\left[
\sup_{t\in[0,T]}\langle 1,\xi^{n_0}_t\rangle 
\right]=
\E\left[
\sup_{t\in[0,T]}\frac{1}{n_0}\langle 1,\mathbbm{x}^{n_0}(t)\rangle 
\right]
<e^{\bar\lambda t}
\]
\end{lemma}
\begin{proof}
Since the birth rate is $\lambda_{\text{b}}\le\bar\lambda $, 
we know that $\langle 1,\mathbbm{x}^{n_0}(t)$ (which is the number of cells) is for any $t$ less 
(in the sense of distribution)
than the number of particles in a simple birth only (Yule) process with branching rate $\bar\lambda$.
\end{proof}

\begin{proof}[Proof of Theorem \ref{thm:well-p-IBM}]
For notational convenience we assume $n_0=1$ and we omit indices $i$ and $n_0$,
besides $\varrho^{n_0}$ to avoid confusion with $\varrho$ from eqn. \eqref{eqn:hybrid}.
The prove for any other $n_0$ follows analogously with $\jj$ replaced by $(i,\jj)$.

Fix $T>0$. We will prove the existence and uniqueness in $[0,T]$.
Since we think about the solution to \eqref{eqn:IBM} pathwise,
the prove will be done for fixed $\omega$.
But firstly, we can neglect the set of probability 0.
By Lemma \ref{lm:Nt},
the number of cells at any finite $t$ is this number is finite  with probability one.
Therefore, if $\Omega_0$ is the set of those $\omega\in\Omega$,
that the microscopic model has infinite number of cells before time $T$,
then $\P(\Omega_0)=0$. Let $\Omega_1$ be the zero measure set for which any of 
the Poisson point processes $\mathcal{N}_{\jj}$ has infinite number of points in $[0,T]\times [0,\bar\lambda]$.

Let us now fix $\omega\in\Omega\setminus(\Omega_0\cup\Omega_1)$ and consider fixed trajectories of 
$\left(W_{\jj}\right)_{\jj\in\JJ}$ 
and $\left(\mathcal{N}_{\jj}\right)_{\jj\in\JJ}$ for this $\omega$.
For $\omega\not\in(\Omega_0\cup\Omega_1)$ there is a finite number of particles born before time $T$
and for any of those particles its Poisson point process has finite number of points in $[0,T]\times [0,\bar\lambda]$,
so there is a finite number, say $\bar n$, of points in all those Poisson  point processes.
Let us denote those times by $\left(t_k,z_k\right)_{k=1,\dots,\bar n}$ in the order of increasing times
and $t_0=0$. In between times $t_k$ the number of cells is constant, so
we solve recursively 
in the intervals  $[t_k,t_{k+1})$, $k=0,1,2,\dots$  a deterministic system
\begin{equation}
\label{eqn:IBMintau}
\begin{cases}
\mathbf{X}_{\jj}(t) = \mathbf{X}_{\jj}(t_k) + 
			\int_{t_k}^t \mathbf{b}(\mathbf{X}_{\jj}(s), \nabla\varrho^{n_0}(s,\mathbf{X}_{\jj}(s)))ds 
											 + \sigma (W_{\jj}(t)-W_{\jj}(t_k)),
\text{ for }\jj\in\JJ_{k},\\
\frac{\partial \varrho^{n_0}(t,\mathbf{x})}{\partial t} = 
D \Delta \varrho^{n_0}(t,\mathbf{x}) -r \varrho^{n_0}(t,\mathbf{x})  + \alpha[K * \xi_t](\mathbf{x}),
\end{cases}
\end{equation}
with 
$
\xi_t = \sum_{{\jj\in\JJ_k}}
	\delta_{X_{\jj}(t)}
$ 
where $\JJ_k$ is the set of indices of cells alive in interval $(t_k,t_{k+1})$,
i.e. $\JJ_0=\{\emptyset\}$ and 
\begin{align*}
\JJ_{k}=&\JJ_{k-1}\cup\{\jj i:\jj\in\JJ_{k-1},\,i=0,1,\,z_k<\lambda_{\text{b}}\bigl(X_{\jj}(t_k^-),\varrho^{n_0}(t_k^-,X_{\jj}(t_k^-))\bigr)\}
\\&\setminus \{\jj :\jj\in\JJ_{k-1},\,z_k<\lambda_{\text{b}}\bigl(X_{\jj}(t_k^-),\varrho^{n_0}(t_k^-,X_{\jj}(t_k^-))\bigr)
+\lambda_{\text{d}}\bigl(X_{\jj}(t_k),\varrho^{n_0}(t_k,X_{\jj}(t_k))\bigr)\}.
\end{align*}
The initial conditions at zero are
$\mathbf{X}_{\jj}(t_0)=\mathbf{X}^{1}_{1,\emptyset}(0)$ and
$\varrho^{n_0}(0,\cdot)=\varrho_0$,
and for $t_k>0$ recursively
$\varrho^{n_0}(t_k,\cdot)=\varrho^{n_0}(t_k^-,\cdot)$
and for $\jj\in\JJ_k$
\begin{equation*}
\mathbf{X}_{\jj}(t_k)=
\begin{cases}
\mathbf{X}_{\jj}(t_k^-), \text{ if }\jj\in\JJ_{k-1},
\\
\mathbf{X}_{\jj^\smalleftarrow}(t_k^-) \text{ otherwise. }
\end{cases}
\end{equation*}
The proof of existence and uniqueness of solutions to \eqref{eqn:IBMintau} on $(t_k,t_{k+1})$
is straightforward and goes e.g. by Piccard type argument. 
\end{proof}

The next proof uses classical methods from \cite{Sznitman1991} and is similar to the proof of 
Proposition 2.3. in \cite{BudhirajaFan2017}.
\begin{proof}[Proof of Theorem \ref{thm:well-p-XM}]
We prove the existence on the interval $[0,T]$. 
The scheme of the proof is the following:  given a fixed function $\varrho:[0,T]\to C^2(\mathbb{R}^d)$
we solve the SDE \eqref{eqn:XMrhoX}-\eqref{eqn:XMrhoM}.
Let $\tilde\varrho$ be a solution to \eqref{eqn:rhotrewritten} with $\mu$ given by \eqref{eqn:def-mu}.
Then we show that operator $P:\varrho\mapsto\tilde\varrho$ is a contraction, 
so there is a unique $\varrho$
satisfying \eqref{eqn:XMrho}, and therefore a unique $X$ and $M$.
We consider the operator $P$
on the space 
$E=\{\varrho\in C_{[0,T]}(C^2_b(\mathbb{R}^d)):
\mathop{\rm Lip}\varrho_t\wedge \mathop{\rm Lip}\nabla\varrho_t\le L \text{ for }t\in[0,T]\}$
where $L$ is the maximum of $L_{\kappa}$ and the Lipschitz coefficient of $\varrho_0$,
with the Bielecki norm 
$\vertiii{\varrho}_\gamma = \max_{t\in [0,T]}e^{-\gamma t}\left(\|\varrho_t\|_\infty+\|\nabla\varrho_t\|_\infty\right)$.
Note that, thanks to properties of heat kernel, $\varrho_t$ given by \eqref{eqn:rhotrewritten}
and $\nabla\varphi _t$
are Lipschitz with coefficient $L$, so $P(E)\subset E$.

To prove that $P$ is contractive, take $\varrho^{(1)},\varrho^{(2)}\in C_{[0,T]}\left(C^2(\mathbb{R}^d)\right)$
and let $(X^{(i)},M^{(i)})$ be the pathwise unique solutions to
\begin{equation}
\label{eqn:XiMi}
\begin{cases}
\mathbf{X}^{(i)}(t) =  \mathbf{X}(0) + \int_0^t \mathbf{b}(\mathbf{X}^{(i)}(s),\nabla\varrho^{(i)}(s,\mathbf{X}^{(i)}(s)))ds + \sigma \,W(t),\\
M^{(i)}(t) = M(0) + \int_0^t \lambda \bigl(\mathbf{X}^{(i)}(s),\varrho^{(i)}(s,\mathbf{X}^{(i)}(s))\bigr)M^{(i)}(s)\,ds,
\end{cases}
\end{equation}
for $i=1,2$. Note that $M^{(i)}(t)\le M(0)e^{\bar\lambda t}=e^{\bar\lambda t}$.
Thus we have
\begin{align*}
|M^{(1)}(t)-M^{(2)}(t)|
\le&
\int_0^t
\Bigl[\bar\lambda|M^{(1)}(s)-M^{(2)}(s)|
+e^{\bar\lambda t}L_\lambda (1+L)\|\mathbf{X}^{(1)}(s)-\mathbf{X}^{(2)}(s)\|
\\&+e^{\bar\lambda t}L_\lambda \|\varrho_s^{(1)}-\varrho_s^{(2)}\|_\infty
\Bigr]\,ds,
\\
\|\mathbf{X}^{(1)}(t)-\mathbf{X}^{(2)}(t)\|\le&
\int_0^t
\Bigl[e^{\bar\lambda}L_{\mathbf{b}} (1+L)\|\mathbf{X}^{(1)}(s)-\mathbf{X}^{(2)}(s)\|
+L_{\mathbf{b}} \|\nabla\varrho_s^{(1)}-\nabla\varrho_s^{(2)}\|_\infty
\Bigr]\,ds.
\end{align*}
So by Gronwall's Lemma we have
\begin{multline}
|M^{(1)}(t)-M^{(2)}(t)|
+
\|\mathbf{X}^{(1)}(t)-\mathbf{X}^{(2)}(t)\|\le\\
c_1\int_0^t\left(\|\varrho_s^{(1)}-\varrho_s^{(2)}\|_\infty
+\|\nabla\varrho_s^{(1)}-\nabla\varrho_s^{(2)}\|_\infty\right)ds.
\end{multline}
where $c_1$ (and similarly $c_2$ to $c_4$ below) is a constant
depend only on the bounds and Lipschitz coefficients of 
the functions $K$, $\nabla K$, $\Delta K$, $\lambda$ and $\mathbf{b}$ and time $T$.
Moreover 
\[
\bigl|
[
 \kappa(\mathbf{y}-X^{(1)}(t))M^{(1)}(t)
- \kappa(\mathbf{y}-X^{(2)}(t))M^{(2)}(t)
]\bigr|
\le
c_2
(|M^{(1)}(t)-M^{(2)}(t)|
+
\|\mathbf{X}^{(1)}(t)-\mathbf{X}^{(2)}(t)\|).
\]
If $\mu^{(i)}$ is given by \eqref{eqn:def-mu} for $(\mathbf{X}^{(i)},M^{(i)})$ then
\[
\kappa*\mu_t^{(i)}(\mathbf{y})=
\E\left[
	\kappa(\mathbf{y}-\mathbf{X}^{(i)}(t))M^{(i)}(t)
\right].
\]
Therefore, if 
$
\tilde\varrho^{(i)}_t(\mathbf{x}) 
	= S_t \varrho_0(\mathbf{x}) 
	+ \alpha\int_0^t S_{t-s}[K * \mu^{(i)}_s](\mathbf{x})\,ds
$
then
\begin{multline*}
|\tilde\varrho^{(1)}_t(\mathbf{x}) 
-\tilde\varrho^{(2)}_t(\mathbf{x}) |
=
	\left|\alpha\int_0^t\int_{\mathbb{R}^d} p(t-s,\mathbf{x},\mathbf{y})
	[K * \mu^{(1)}_s-K * \mu^{(2)}_s](\mathbf{y})\,d\mathbf{y}\,ds\right|
\\\le
c_3	\int_0^t \int_0^r \left(\|\varrho_s^{(1)}-\varrho_s^{(2)}\|_\infty
+\|\nabla\varrho_s^{(1)}-\nabla\varrho_s^{(2)}\|_\infty\right)\,ds\,dr 
\end{multline*}
and likewise
\begin{multline*}
|\nabla\tilde\varrho^{(1)}_t(\mathbf{x}) 
-\nabla\tilde\varrho^{(2)}_t(\mathbf{x}) |
\le
c_4	\int_0^t \int_0^r \left(\|\varrho_s^{(1)}-\varrho_s^{(2)}\|_\infty
+\|\nabla\varrho_s^{(1)}-\nabla\varrho_s^{(2)}\|_\infty\right)\,ds\,dr .
\end{multline*}
Therefore,
\begin{multline*}
\vertiii{\tilde\varrho^{(1)}-\tilde\varrho^{(2)}}
\le
(c_3+c_4) \int_0^t \int_0^r e^{-\gamma (t - s)} \vertiii{\tilde\varrho^{(1)}-\tilde\varrho^{(2)}} ds dr 
= \\(c_3+c_4) \frac{e^{-\gamma t} (-\gamma t + e^{\gamma t} - 1)}{\gamma^2}
\vertiii{\tilde\varrho^{(1)}_s-\tilde\varrho^{(2)}_s},
\end{multline*}
so, for sufficiently large $\gamma$,  $P$ is contractive in $\vertiii{\cdot}_\gamma$.	
\end{proof}

Now we are ready to prove the well-posedness of the mean-field model.
\begin{proof}[Proof of Theorem \ref{thm:well-p-hybrid}]
The scheme will be the following. For any function $\varrho\in C_{C^1_b(\mathbb{R}^d)}[0,T]$
we notice the existence and uniqueness of branching diffusion process given by \eqref{eqn:hybridX+bd}
(a proof can be done as in the proof of Theorem \ref{thm:well-p-IBM}). Then we show that $\bar\mu$ given by \eqref{eqn:barmu} is equal to $\mu$  
given by \eqref{eqn:def-mu} for $(\mathbf{X},M)$ obtained as solution to \eqref{eqn:XiMi} with the same 
$\varrho$. That means that, by Theorem \ref{thm:well-p-XM}, there exists a unique $\varrho$ such that  
\eqref{eqn:hybrid} is satisfied.

To this aim, fix $\varrho:[0,T]\to C^2_b(\mathbb{R}^d)$ continuous in time and 
let process $\bar{\mathbbm{x}}(t)={\bigl(\mathbf{\bar X}_{\jj}(t)\bigr)_{\jj\in\JJ}}$, $t\in[0,T]$ 
be the  solution to \eqref{eqn:hybridX+bd} with this fixed $\varrho$.
For $\varphi:\mathbb{R}^d\cup\{\phi\}\to\mathbb{R}$ such that $\varphi (\phi)=0$ 
and $\varphi|_{\mathbb{R}^d}\in C^2_b(\mathbb{R}^d)$
 denote 
\[
\langle \varphi ,\mathbbm{\bar x}(t)\rangle =
\sum_{\jj\in\JJ}\varphi (\mathbf{\bar X}_\jj(t)).
\]
Note that this sum is finite.
By It\^o's Lemma we have
\begin{align*}
\langle \varphi ,\mathbbm{\bar x}(t)\rangle = &
\langle \varphi ,\mathbbm{x}(0)\rangle
+
\int_0^t
\left\langle 
	\mathbf{b}(\cdot, \nabla\varrho(t,\cdot))
		\nabla\varphi (\cdot)+\frac{\sigma ^2}{2}\Delta\varphi (\cdot),
	\mathbbm{\bar x}(s)
\right\rangle ds
+
\sigma \sum_{\jj\in\JJ} \int_0^t
 \nabla\varphi (\mathbf{\bar X}_\jj(s)) dW_{1,\jj}(s)
\\&
-\sum_{\jj\in\JJ} \int_0^t
\varphi(X_\jj(s^-))\II_{[0,\lambda_\text{b}(X_\jj(s^-),\rho(s^-,\mathbf{\bar X}_\jj(s^-)))+
								\lambda_\text{d}(X_\jj(s^-),\rho(s^-,\mathbf{\bar X}_\jj(s^-)))]}(z) \mathcal{N}_{1,\jj}(ds,dz) 
\\&+
\sum_{\jj\in\JJ} \int_0^t
2\varphi(\mathbf{\bar X}_\jj(s^-))\II_{[0,\lambda_\text{b}(X_\jj(s^-),\rho(s^-,\mathbf{\bar X}_\jj(s^-)))]}(z) \mathcal{N}_{1,\jj}(ds,dz) 
\end{align*}
and thus
\begin{equation*}
\E \langle \varphi ,\mathbbm{x}(t)\rangle = 
\E\langle \varphi ,\mathbbm{x}(0)\rangle
+
\E\int_0^t
\left\langle 
	B_{\varrho,s}\varphi ,\mathbbm{x}(s)
\right\rangle 
+
\left\langle 
	\lambda(\cdot ,\varrho(s,\cdot )) \varphi(\cdot )  ,\mathbbm{x}(s)
\right\rangle ds,
\end{equation*}
where
$B_{\varrho,s}\varphi(\mathbf{x}) =	\frac{\sigma ^2}{2}\Delta\varphi (\mathbf{x})
+\mathbf{b}(\mathbf{x}, \nabla\varrho(s,\mathbf{x}))
\nabla\varphi (\mathbf{x})$ 
for $\mathbf{x}\in\mathbb{R}^d$
and $B_{\varrho,s}\varphi(\phi)=0$,
and
$\lambda=\lambda_\text{b}-\lambda_\text{d}$.

Let 
$
\bar\xi_t = \sum_{{\jj\in\JJ}}
	\II_{\mathbb{R}^d}(\mathbb{\bar X}_{\jj}(t))\delta_{\mathbb{\bar X}_{\jj}(t)}
$ 
and let $\langle \varphi,\bar\xi_t  \rangle =\int_{\mathbb{R}^d} \varphi (\mathbf{x}) \bar\xi_t(d \mathbf{x})$.
Then $\langle \varphi,\bar\xi_t  \rangle =\langle \varphi, \mathbbm{\bar x}(t) \rangle$ and
\begin{align*}
 \langle \varphi ,\E\bar\xi_t\rangle = &
\langle \varphi ,\E\bar\xi_0\rangle
+
\int_0^t
\left\langle 
B_{\varrho,s}\varphi ,\E\bar\xi_s
\right\rangle 
+
\left\langle 
	\lambda(\cdot ,\varrho(s,\cdot )) \varphi(\cdot ) ,\E\bar\xi_s
\right\rangle ds,
\end{align*}
which means that
\begin{equation}
\label{eqn:weekfor<barmu>}
\langle \varphi ,\bar\mu_t\rangle = 
\langle \varphi ,\bar\mu_0\rangle
+
\int_0^t
\left\langle 
	B_{\varrho,s}\varphi ,\bar\mu_s
\right\rangle 
+
\left\langle 
	\lambda(\cdot ,\varrho(s,\cdot )) \varphi(\cdot ) ,\bar\mu_s
\right\rangle ds.
\end{equation}
This is the week version of Equation \eqref{eqn:macrop} and it is well known that it admits a unique solution
which is absolutely continuous with respect to Lebesgue measure for $t>0$ 
even if $\bar\mu_0$ is not.

Let now  $(\mathbf{X},M)$ be a process obtained as a solution to \eqref{eqn:XiMi} with given $\varrho$
and $\varphi\in C^2_b(\mathbb{R}^d)$.
Then, by It\^o formula we have
\begin{align*}
 \varphi(\mathbf{X}(t))M(t)
=&
\varphi (\mathbf{X}(0))M(0)
+\int_0^t
\biggl[
	M(s)
	\nabla \varphi(\mathbf{X}(s))\cdot \mathbf{b}(\mathbf{X}(s),\nabla\varrho(s,\mathbf{X}(s)))
\\&
	+\varphi(\mathbf{X}(s))\lambda \bigl(\mathbf{X}(s),\varrho(s,\mathbf{X}(s))\bigr)M(s)
	+\frac{\sigma^2}{2}\Delta \varphi(\mathbf{X}(s))
\biggl]\,ds\\&
- \int_0^t 
\nabla \varphi(\mathbf{X}(s))\cdot \mathbf{b}(\mathbf{X}(s), \nabla\varrho(s,\mathbf{X}(s)))
dW(s)
\end{align*}
Note that for $\mu_t$ given by \eqref{eqn:def-mu} we have 
$\langle \varphi ,\mu_t\rangle =\int_{\mathbb{R}^d}\varphi (\mathbf{x})\mu_t(d\mathbf{x})=\E[\varphi (\mathbf{X}(t))M(t)]$.
Taking expectation on both sides of the equation above we get
\begin{equation}
\langle \varphi ,\mu_t\rangle =
\int_0^t
\left\langle 
B_{\varrho,s}\varphi(\cdot )+
	\lambda(\cdot ,\varrho(s,\cdot )) \varphi(\cdot ) ,\mu_s
\right\rangle ds.
\end{equation}
which is exactly the same as \eqref{eqn:weekfor<barmu>}.
\end{proof}

\subsection{Proof of convergence}


\begin{proof}[Proof of Theorem \ref{thm:convinD}]
In order to prove the point (i) we check that the sequence $\xi^{n_0}_t$ is tight on $D_{\mathcal{M}[0,t]}$
and then we check that the limit of any subsequence has to coincide with $\bar\mu_t$.
Similarly, we prove tightness of $\varrho^{n_0}_t$ in $C_{C^1_b(\mathbb{R}^d)}[0,T]$ 
and check that the limit has to satisfy \eqref{eqn:rhotrewritten}.

\textbf{Tightness of $\{\xi^{n_0}\}_{n_0\in\mathbb{N}}$}. 
The process $\xi^{n_0}$ has values in the space $\mathcal{M}$ of finite positive Radon measures on $\mathbb{R}^d$. 
Note that the $\mathcal{M}$ with the vague convergence topology can be metrizable, eg. with metric
\[
d_{\mathcal{M}}(\mu,\nu)
=
\sum_{k=1}^\infty \frac{1}{2^k}\min\{1, \langle \varphi_k,\mu-\nu\rangle \}
\]
with some sequence  $\varphi_k\in C_c(\mathbb{R}^d)$ (cf. \cite[Section §31]{Bauer2001})
in such a way that $\mathcal{M}$ is complete. Moreover,
set $H\subset\mathcal{M}$ is  vaguely relatively compact if and only if
\[
\sup_{\mu\in H}|\langle f,\mu\rangle| \quad \text{ for all }f\in C_c(\mathbb{R}^d),
\]
where one can take $(\varphi_k)_{k\in\mathbb{N}}$ instead of all $f\in C_c(\mathbb{R}^d)$.
Therefore, Proposition 1.7 from \cite[Chapter 4]{KipnisLandim1999}
holds for processes with values in $(\mathcal{M},d_{\mathcal{M}})$.
Now, thanks to Aldous criterion (see, eg. \cite[Chapter VI, Theorem 4.5]{JacodShiryaev1987})
for the relative compactness of $\{\xi^{n_0}_t\}$
it suffices to check for all $\varphi_k$
that for any $\varepsilon >0$ there exists $M>0$ such that
\begin{equation}
\label{eqn:aldous1}
\P\left(
	\langle \varphi_k,\xi^{n_0}_t\rangle >M
\right)<\varepsilon ,
\text{ for all }t\in[0,T]
\text{ and }n_0\in\mathbb{N}
\end{equation}
and
\begin{equation}
\label{eqn:aldous2}
\lim_{\gamma \to 0} \limsup_{n_0\to \infty} 
\sup_{\tau\in\mathcal{T}_T,\theta<\gamma}
\P\left(
|\langle \varphi_k , \xi^{n_0}_{\tau+\theta}\rangle 
-\langle \varphi_k , \xi^{n_0}_\tau\rangle |>\varepsilon 
\right)=0,
\end{equation}
where $\mathcal{T}_T$ is the set
of all stopping times bounded by $T$.
Note that \eqref{eqn:aldous1} follows by Markov's inequality from Lemma \ref{lm:Nt}.
To prove \eqref{eqn:aldous2}, using It\^o's Lemma we calculate
\begin{align*}
\langle \varphi_k ,\xi^{n_0}_t\rangle =&
\langle \varphi_k ,\tfrac{1}{n_0}\mathbbm{x}^{n_0}(t)\rangle = 
\langle \varphi_k ,\tfrac{1}{n_0}\mathbbm{x}^{n_0}(0)\rangle
\\&+
\frac{1}{n_0}\int_0^t
\left(
	\sum_{i=1}^{n_0}\sum_{\jj\in\JJ} 
		\mathbf{b}(\mathbf{X}^{n_0}_{i,\jj}(s), \nabla\varrho^{n_0}(t,\mathbf{X}^{n_0}_{i,\jj}(s)))
		\nabla\varphi_k (\mathbf{X}^{n_0}_{i,\jj}(s))
	+
	\frac{\sigma^2}{2}\Delta \varphi_k (\mathbf{X}^{n_0}_{i,\jj}(s))
\right)
ds
\\&+
\frac{1}{n_0}\int_0^t
\sigma \sum_{i=1}^{n_0}\sum_{\jj\in\JJ} \nabla\varphi_k (\mathbf{X}^{n_0}_{i,\jj}(s)) dW_{i,\jj}(s)
+\frac{1}{n_0}\sum_{i=1}^{n_0}\sum_{s\le t}\left(\langle \varphi_k (\mathbf{X}^{n_0}_{i,\jj}(s)) -\varphi_k (\mathbf{X}^{n_0}_{i,\jj}(s^-))\rangle \right)
\\=&
\langle \varphi_k ,\xi^{n_0}_0\rangle
+
\int_0^t
\left\langle 
	\mathbf{b}(\cdot, \nabla\varrho^{n_0}(s,\cdot))
		\nabla\varphi_k (\cdot)+\frac{\sigma ^2}{2}\Delta\varphi_k (\cdot)
		+\lambda(\cdot ,\varrho^{n_0}(s,\cdot )) \varphi_k(\cdot ),
	\xi^{n_0}_s
\right\rangle ds
\\
M^{n_0}_{1,k}(t)\begin{cases}
\\[2ex] 
\end{cases}
&
+
\frac{\sigma}{n_0} \sum_{i=1}^{n_0}\sum_{\jj\in\JJ} \int_0^t
 \nabla\varphi_k (\mathbf{X}^{n_0}_{i,\jj}(s)) dW_{i,\jj}(s)
\\
M^{n_0}_{2,k}(t)\begin{cases}
\\ \\ \\ \\ \\ \\
\end{cases}
&
\begin{aligned}
&
-\frac{1}{n_0}\sum_{i=1}^{n_0}\sum_{\jj\in\JJ} \int_0^t 
\varphi_k(\mathbf{X}^{n_0}_{i,\jj})\II_{[0,\lambda_\text{b}(X_\jj(s^-),\rho^{n_0}(s^-,\mathbf{X}^{n_0}_{i,\jj}(s^-)))+
								\lambda_\text{d}(X_\jj(s^-),\rho^{n_0}(s^-,\mathbf{X}^{n_0}_{i,\jj}(s^-)))]}(z) \mathcal{N}_{i,\jj}(ds,dz) 
\\&+
\frac{1}{n_0}\sum_{i=1}^{n_0}\sum_{\jj\in\JJ} \int_0^t
2\varphi_k(\mathbf{X}^{n_0}_{i,\jj}(s^-))\II_{[0,\lambda_\text{b}(X_\jj(s^-),\rho^{n_0}(s^-,\mathbf{X}^{n_0}_{i,\jj}(s^-)))]}(z) \mathcal{N}_{i,\jj}(ds,dz) 
\\&
-
\int_0^t
\left\langle 
	\lambda(\cdot ,\varrho^{n_0}(s,\cdot )) \varphi_k(\cdot ),
	\xi^{n_0}_s
\right\rangle ds,
\end{aligned}
\end{align*}
where $M^{n_0}_{1,k}(t)$ and $M^{n_0}_{1,k}(t)$ are martingales.
Therefore,
\begin{align*}
\langle \varphi_k ,\xi^{n_0}_{\tau+\theta}-\xi^{n_0}_\tau\rangle =&
\int_\tau^{\tau+\theta}
\left\langle 
	\mathbf{b}(\cdot, \nabla\varrho^{n_0}(s,\cdot))
		\nabla\varphi_k (\cdot)+\frac{\sigma ^2}{2}\Delta\varphi_k (\cdot)
		+\lambda(\cdot ,\varrho^{n_0}(s,\cdot )) \varphi_k(\cdot ),
	\xi^{n_0}_s
\right\rangle ds
\\&+
M^{n_0}_{1,k}(\tau+\theta)-M^{n_0}_{1,k}(\tau)+M^{n_0}_{2,k}(\tau+\theta)-M^{n_0}_{2,k}(\tau)
\end{align*}
The integral over $ds$ can be estimated by 
a constant times $\theta\sup_{s\in[0,T]}\langle 1,\xi^{n_0}_s\rangle $,
so, since $\theta\le\gamma$,  by Lemma \ref{lm:Nt} and Markov inequality, 
probability that it is greater then $\varepsilon $
goes to zero as $\gamma\to0$.
By It\^o's Lemma,
\[
(M^{n_0}_{1,k}(t))^2 = 
\frac{\sigma^2}{n_0^2} \sum_{i=1}^{n_0}\sum_{\jj\in\JJ} \int_0^t
 \left(\nabla\varphi_k (\mathbf{X}^{n_0}_{i,\jj}(s))\right)^2 ds + \tilde M^{n_0}_{1,k}(t),
\]
where $\tilde M^{n_0}_{1,k}(t)$ is a martingale, so
\begin{multline}
\label{eqn:EM1}
\E\left[\left(M^{n_0}_{1,k}(\tau+\theta)-M^{n_0}_{1,k}(\tau)\right)^2\right]=
\E\left[\left(M^{n_0}_{1,k}(\tau+\theta)\right)^2\right]-\E\left[\left(M^{n_0}_{1,k}(\tau)\right)^2\right]=\\
\E\left[\frac{\sigma^2}{n_0^2} \sum_{i=1}^{n_0}\sum_{\jj\in\JJ} \int_\tau^{\tau+\theta}
 \left(\nabla\varphi_k (\mathbf{X}^{n_0}_{i,\jj}(s))\right)^2 ds \right]
\le \theta \frac{\sigma^2}{n_0} \|\nabla\varphi_k \|_\infty 
\E\left[\sup_{s\in[0,T]}\frac{\langle 1,\xi^{n_0}\rangle }{n_0}\right].
\end{multline}
Similarly,
using  It\^o's formula again we get
\begin{align*}
(M^{n_0}_{2,k}(t))^2 = &
-\int_0^t
2M^{n_0}_{2,k}(s)
\left\langle 
		\lambda(\cdot ,\varrho^{n_0}(s,\cdot )) \varphi_k(\cdot )  ,\xi^{n_0}_s
\right\rangle
ds
\\&-
\sum_{i=1}^{n_0}\sum_{\jj\in\JJ} \int_0^t
 \left(\frac{1}{n_0^2}\varphi_k(\mathbf{X}^{n_0}_{i,\jj}(s^-))^2
 -\frac{2}{n_0}M^{n_0}_{2,k}(s)
 \varphi_k(\mathbf{X}^{n_0}_{i,\jj}(s^-))
 \right)
\\&\qquad \qquad \qquad  \times 
 \II_{[0,\lambda_\text{b}(\mathbf{X}^{n_0}_{i,\jj}(s^-),\varrho^{n_0}(s^-,\mathbf{X}^{n_0}_{i,\jj}(s^-)))+
								\lambda_\text{d}(\mathbf{X}^{n_0}_{i,\jj}(s^-),\varrho^{n_0}(s^-,\mathbf{X}^{n_0}_{i,\jj}(s^-)))]}(z) \mathcal{N}_{i,\jj}(ds,dz) 
\\&+
\sum_{i=1}^{n_0}\sum_{\jj\in\JJ} \int_0^t
 \left(\frac{4}{n_0^2}\varphi_k\left(\mathbf{X}^{n_0}_{i,\jj}(s^-)\right)^2
 +\frac{4}{n_0}M^{n_0}_{2,k}(s)
\varphi_k\left(\mathbf{X}^{n_0}_{i,\jj}(s^-)\right)
 \right)
\\&\qquad \qquad \qquad  \times 
 \II_{[0,\lambda_\text{b}(\mathbf{X}^{n_0}_{i,\jj}(s^-),\varrho^{n_0}(s^-,\mathbf{X}^{n_0}_{i,\jj}(s^-)))]}(z) \mathcal{N}_{i,\jj}(ds,dz) 
\\=&
\frac{1}{n_0}\int_0^t
\left\langle [3\lambda_\text{b}(\cdot ,\varrho^{n_0}(s^-,\cdot ))-
								\lambda_\text{d}(\cdot ,\varrho^{n_0}(s^-,\cdot ))]\varphi_k(\cdot )^2,
		\xi^{n_0}_s\right\rangle 
ds
+\tilde M^{n_0}_{2,k}(t).
\end{align*}
where $\tilde M^{n_0}_{2,k}(t)$ is a martingale, thus again
\begin{multline}
\label{eqn:EM2}
\E\left[\left(M^{n_0}_{2,k}(\tau+\theta)-M^{n_0}_{2,k}(\tau)\right)^2\right]
\le \theta \frac{3\bar\lambda}{n_0} \|\varphi_k^2 \|_\infty 
\E\left[\sup_{s\in[0,T]}\frac{\langle 1,\xi^{n_0}\rangle }{n_0}\right],
\end{multline}
which completes the proof of \eqref{eqn:aldous2}.


\textbf{Tightness of $\{\varrho^{n_0}\}_{n_0\in\mathbb{N}}$.}
Although it would be sufficient to use the topology of locally uniform convergence, 
we need a stronger convergence in the next  proof, so
let us consider $C^1_b(\mathbb{R}^d)$ with the topology 
of locally  uniform convergence of function and its derivative,
that is we us a norm 
\begin{equation}
\|f\|_{C^2_b}=
\sum_{R=1}^\infty \frac{1}{2^R}\left(\sup_{x\in B(0,R)}|f(x)|+\sup_{x\in B(0,R)}\|\nabla f(x)\|\right).
\end{equation}
Note that a set
$
K_M = \{f\in C^1_b(\mathbb{R}^d):
	 \|f\|_\infty\le M,\|\nabla f\|_\infty\le M, \|\text{Hess} f\|_\infty\le M\},$
where $\|\text{Hess} f\|_\infty = \sup_{\mathbf{x}\in\mathbb{R}^d}\max_{1\le i,j \le d}|\partial_i \partial _j f(\mathbf{x})|$,
is relatively compact in this norm.
Using the version of Ascoli Theorem (see eg. \cite[Theorem 47.1]{Munkres2000})
we know that a family $\mathcal{K}_M\subset \mathcal{C}_T=C_{C^1_b(\mathbb{R}^d)}[0,T]$ 
 of functions $g$,
which are equicontinuous in $t$ and such that $\{g(t):g\in\mathcal{K}_M\}\subset K_M$  for each $t\in[0,T]$,
is relatively compact.
Now, in order to prove tightness of $\varrho^{n_0}_t$, we need to check that 
for any $\varepsilon >0$ there exists $M>0$ such that
$\P(\varrho^{n_0}\in\mathcal{K}_M)>1-\varepsilon $.
To that end, recall that $\varrho^{n_0}_t=S_t\rho_0+\alpha\int_0^t S_{t-s}\kappa*\xi^{n_0}_s(x) ds$
where the first summand is continuous in $\mathcal{C}_T$ and the latter is Lipschitz with probability $1-\varepsilon $,
because
\[
\left|\alpha\int_t^{t+\theta}
 S_{t+\theta-s}[\kappa * \xi^{n_0}_s](\mathbf{x})\,ds\right|
\le \theta \alpha \|\kappa\|_\infty \sup_{s\in[0,T]}\langle 1,\xi^{n_0}_s\rangle
\]
and
\[
\left|\alpha\nabla\int_t^{t+\theta}
 S_{t+\theta-s}[\kappa * \xi^{n_0}_s](\mathbf{x})\,ds\right|
\le \theta \alpha \|\nabla\kappa\|_\infty \sup_{s\in[0,T]}\langle 1,\xi^{n_0}_s\rangle.
\]
Moreover, we have
\[
|\varrho^{n_0}_t(x)|\le
|S_t\rho_0(x)|+\sup_{s\in[0,T]}\|\kappa*\xi^{n_0}_s\|_\infty
\le
\|\rho_0\|_\infty+t\|\kappa\|_\infty \sup_{s\in[0,T]}\langle 1,\xi^{n_0}_s\rangle,
\]
similarly
\begin{equation}
\label{eqn:rhon0Lip}
|\nabla \varrho^{n_0}_t(x)|
\le
\|\nabla\rho_0\|_\infty+t\,\|\nabla\kappa\|_\infty \sup_{s\in[0,T]}\langle 1,\xi^{n_0}_s\rangle,
\end{equation}
and further
\begin{equation}
\label{eqn:Hessrhon}
|\partial _i\partial _j \varrho^{n_0}_t(x)|
\le
\|\partial _i\partial _j\rho_0\|_\infty+t\,\|\partial _i\partial _j\kappa\|_\infty \sup_{s\in[0,T]}\langle 1,\xi^{n_0}_s\rangle .
\end{equation}
These estimates with Lemma \ref{lm:Nt} and Markov's inequality complete the proof of tightness.

\textbf{Identification of the limit.} By the similar estimates as in \eqref{eqn:EM1} and  \eqref{eqn:EM2}
we get that the limit has to satisfy \eqref{eqn:weekfor<barmu>} for any $\varphi \in C^2_c(\mathbb{R}^d)$ 
and \eqref{eqn:rhotrewritten} which admit a unique solution.
\end{proof}

\subsection{Proof of Theorem \ref{thm:conv}}
Fix $T>0$ and $\varepsilon >0$.
We use here the coupling of ${\mathbbm{x}}^{n_0}_1(t)$ and $\bar{\mathbbm{x}}(t)$
obtained by using the same processes $W_{1,\jj}$ and $\mathcal{N}_{1,\jj}$,
and the fact from Theorem \ref{thm:convinD} that  $\varrho^{n_0}$ converges in probability to $\varrho $.

Let $\bar\sigma_{\jj}$ and $\bar\tau_{\jj}$ denote times of birth and death, respectively,
of the $\jj$-th particle of $\bar{\mathbbm{x}}(t)$
and let us construct such a process  $\mathbbm{\tilde x}^{n_0}_1$
that its $\jj$-th particle lives
from $\bar\sigma_{\jj}$ to $\bar\tau_{\jj}$
and moves during this time according to the equation
\begin{equation*}
\mathbf{\tilde X}^{n_0}_{1,\jj}(t) = \mathbf{\tilde X}^{n_0}_{1,\jj}\left(\bar\sigma_{\jj}\right)
	+ \int_{\bar\sigma_{\jj}}^t \mathbf{b}(\mathbf{\tilde X}^{n_0}_{1,\jj}(s), 
		\nabla\varrho^{n_0}(s,\mathbf{\tilde X}^{n_0}_{i,\jj}(s)))\,ds 
	+ \sigma \left(W_{1,\jj}(t)-W_{1,\jj}(\bar\sigma_{\jj})\right),
\end{equation*}
for $t\in \left[\bar\sigma_{\jj},\bar\tau_{\jj}\right)$.
It means that the particles of $\mathbbm{\tilde x}^{n_0}_1$ die and are born in the same times as particles of $\mathbbm{\bar x}$
but their dynamics is the same as the dynamics of $\mathbbm{x}^{n_0}_1$.
The idea is to prove, that for $n_0$ large enough with high probability  
$\mathbbm{\tilde x}^{n_0}_1$ is close to $\mathbbm{\bar x}$
 and equal to $\mathbbm{x}^{n_0}_1$.
\begin{lemma}
\label{lm:Omegaeta}
Fix $T>0$. For any $\eta >0$ we can find a set $\Omega_\eta $
such that $\P(\Omega_\eta )>1-\eta $
and 
\begin{enumerate}
\renewcommand{\labelenumi}{\textbf{\theenumi}}
\renewcommand{\theenumi}{(\roman{enumi})}
\item
there exists $\bar N>0$  such that $\sup_{t\in[0,T]}\langle 1,\bar{\mathbbm{x}}(t)\rangle \le \bar N$
i.e. there is at most $\bar N$ particles of $\bar{\mathbbm{x}}(t)$ alive to time $T$
\item
there exists  $R>0$ such that all particles of $\bar{\mathbbm{x}}(t)$ 
live in the ball of radius $R$, i.e.
\[\sup_{t\in [0,T]} \max_{\jj\in\JJ}\|\mathbf{\bar X}_{\jj}\|\le R\]
(there is at most $\bar N$ particles in the maximum).
\end{enumerate}
for $\omega \in \Omega_\eta $.
\end{lemma}
\begin{proof}
The first point is a simple consequence of Lemma \ref{lm:Nt}. Once we have finite nuber of particles, their positions are described by a
finite number of It\^o equations with bounded drift $\mathbf{b}$, so (ii) obviously follows.
\end{proof}

 Let us denote by $\JJ_{\bar N}$ the (finite) set of all indices of the length at most $\bar N$.
 Apparently, if there were not more than $\bar N$ particles of $\bar{\mathbbm{x}}(t)$ up to time $T$,
 then their indices are in $\JJ_{\bar N}$.
\begin{lemma}
\label{lm:delta}
For any $\eta >0$ there exists $\delta>0$ such that
if 
\begin{equation}
\label{eqn:supRrhoDrho}
\sup_{t\in[0,T]}\left(
\sup_{\mathbf{x}\in B(0,R)}\|\varrho ^{n_0}(\mathbf{x})-\varrho(\mathbf{x}) \|
+
\sup_{\mathbf{x}\in B(0,R)}\|\nabla\varrho ^{n_0}(\mathbf{x})-\nabla\varrho(\mathbf{x}) \|
\right)<\delta,
\end{equation}
where $R$ is from Lemma \ref{lm:Omegaeta}, then
the probability that 
process
$\mathbbm{\tilde x}^{n_0}_1$ is different than $\mathbbm{x}^{n_0}_1$
is less than $3\eta $.
\end{lemma}
\begin{proof}
Throughout the proof we assume we are in $\Omega_\eta$ from Lemma \ref{lm:Omegaeta},
$\bar N$ and $R$ are as in Lemma \ref{lm:Omegaeta} and every 
particle of $\mathbbm{\bar x}$ alive during $[0,T]$ has index $\jj\in\JJ_{\bar N}$.
The movement of  particles of $\mathbbm{\tilde x}^{n_0}_1$ and $\mathbbm{x}^{n_0}_1$
is given by the same equation, so the processes are different if and only if
any time of birth or death is different. Recall that $\sigma^{n_0}_{1,\jj}$ and $\tau^{n_0}_{1,\jj}$
are given by \eqref{eqn:sigmadef} and \eqref{eqn:taudef}, and $\bar\sigma_{\jj}$ and $\bar\tau_{\jj}$
analogously with $\mathbf{X}^{n_0}_{i,\jj}$ and $\varrho ^{n_0}$ replaced by 
$\mathbf{\bar X}_{\jj}$ and $\varrho$.
Therefore, if for every $\jj\in\JJ_{\bar N}$ there are no points of $\mathcal{N}_{1,\jj}$
in between 
 $\lambda\bigl(\mathbf{X}^{n_0}_{i,\jj}(t),\varrho^{n_0}(t,\mathbf{X}^{n_0}_{i,\jj}(t))\bigr)$
and $\lambda\bigl(\mathbf{\bar X}_{\jj}(t),\varrho(t,\mathbf{\bar X}_{\jj}(t))\bigr)$ 
nor between
 $\lambda_{\text{b}}\bigl(\mathbf{X}^{n_0}_{i,\jj}(t),\varrho^{n_0}(t,\mathbf{X}^{n_0}_{i,\jj}(t))\bigr)$
and $\lambda_{\text{b}}\bigl(\mathbf{\bar X}_{\jj}(t),\varrho(t,\mathbf{\bar X}_{\jj}(t))\bigr)$
for $t\in\left[\bar\sigma_{\jj},\bar\tau_{\jj}\right)$, then
for all $\jj\in\JJ_{\bar N}$ we have
$\sigma^{n_0}_{1,\jj}=\bar\sigma_{\jj}$, $\tau^{n_0}_{1,\jj0}=\bar\tau_{1,\jj0}$  
and  $\tau^{n_0}_{1,\jj1}=\bar\tau_{1,\jj1}$.

Note that  for $t\in\left[\bar\sigma_{\jj},\bar\tau_{\jj}\right)$ we have
\begin{align*}
\mathbf{\tilde X}^{n_0}_{1,\jj}(t)-\mathbf{\bar X}_{\jj}(t)=&
\mathbf{\tilde X}^{n_0}_{1,\jj}\left(\bar\sigma_{\jj}\right)
-\mathbf{\bar X}_{\jj}\left(\bar\sigma_{\jj}\right)
\\&
+ \int_{\bar\sigma_{\jj}}^t 
	\left[\mathbf{b}(\mathbf{\tilde X}^{n_0}_{1,\jj}(s), 
		\nabla\varrho^{n_0}(s,\mathbf{\tilde X}^{n_0}_{i,\jj}(s)))
		-\mathbf{b}(\mathbf{\bar X}_{\jj}(s), 
		\nabla\varrho(s,\mathbf{\bar X}_{\jj}(s)))
	\right]
ds\end{align*}
where
\begin{multline*}
\left| 
	\mathbf{b}(\mathbf{\tilde X}^{n_0}_{1,\jj}(s), 
	\nabla\varrho^{n_0}(s,\mathbf{\tilde X}^{n_0}_{i,\jj}(s)))
	-\mathbf{b}(\mathbf{\bar X}_{\jj}(s), 
	\nabla\varrho(s,\mathbf{\bar X}_{\jj}(s)))
\right|
\le\\
L_{\mathbf{b}}
\biggl(
	|\mathbf{\tilde X}^{n_0}_{i,\jj}(s)-\mathbf{\bar X}_{\jj}(s))|
	+
	\left|
		\nabla\varrho^{n_0}(s,\mathbf{\tilde X}^{n_0}_{i,\jj}(s)))
		-\nabla\varrho^{n_0}(s,\mathbf{\bar X}_{\jj}(s)))
	\right|	
	+\\
	\left|\nabla\varrho^{n_0}(s,\mathbf{\bar X}_{\jj}(s)))
	-\nabla\varrho(s,\mathbf{\bar X}_{\jj}(s)))
	\right|
\biggr).
\end{multline*}
By \eqref{eqn:rhon0Lip} and \eqref{eqn:Hessrhon},  
functions $\varrho^{n_0}$ and $\nabla\varrho^{n_0}$ are Lipschitz 
with some constant $L_{\varrho,\eta} $ with probability greater  then $1-\eta$,
so if we denote 
$\Delta_{1}(t) = \max_{\jj\in\JJ_{\bar N}}
\left|\mathbf{\tilde X}^{n_0}_{1,\jj}(t)-\mathbf{\bar X}_{\jj}(t)\right|$
then by \eqref{eqn:supRrhoDrho}
\begin{align*}
\Delta_{1}(t)\le&
L_{\mathbf{b}}\int_0^t
	\Bigr((1+L_{\varrho,\eta} )
		\Delta_{1}(s)
	+\sup_{t \in [0,T]}\sup_{\mathbf{x}\in B(0,R)}\left\|
		\nabla\varrho^{n_0}_t(\mathbf{x})-\nabla\varrho_t(\mathbf{x})
	\right\|\Bigr)
ds
\\\le&
L_{\mathbf{b}}\int_0^t
\Bigl(	(1+L_{\varrho,\eta} )
		\Delta_{1}(s)
	+ \delta
\Bigr)ds.
\end{align*}
By Gronwall's inequality we have  
$\Delta_{1}(t)\le\delta\, t \,L_{\mathbf{F}} e^{L_{\mathbf{F}}(1+L_{\varrho,\eta} )t}$ so
\begin{equation}
\label{eqn:Delta}
\sup_{t\in[0,T]}\Delta_{1}(t)\le 
c^{(1)}_{T,\eta}\delta.
\end{equation}
and thus
\begin{equation}
\label{eqn:Deltalambda}
\left|\lambda\bigl(\mathbf{X}^{n_0}_{i,\jj}(t),\varrho^{n_0}(t,\mathbf{X}^{n_0}_{i,\jj}(t))\bigr)
-\lambda\bigl(\mathbf{\bar X}_{\jj}(t),\varrho(t,\mathbf{\bar X}_{\jj}(t))\bigr)\right|
\le L_{\lambda}((1+L_{\varrho,\eta}) \Delta_{1}(t)+\delta) 
\le c^{(2)}_{T,\eta} \delta,
\end{equation}
where $c^{(i)}_{T,\eta}$, $i=1,2$ depend only on $T$ and $L_{\varrho,\eta}$.
Similar estimate holds for $\lambda_{\text{b}}$.
Let
\begin{equation*}
A_\jj=
\biggl\{(t,z)\subset[0,T]\times [0,\bar\lambda]:
	t\in\left[\bar\sigma_{\jj},\bar\tau_{\jj}\right),
	z\in[\hat\lambda^{n_0}_{\text{b}}(t),\check\lambda^{n_0}_{\text{b}}(t)]
	\text{ or }
	z\in[\hat\lambda^{n_0}(t),\check\lambda^{n_0}(t)]
\biggr\},
\end{equation*}
where
$\hat\lambda^{n_0}(t)=$ $\min
\left\{
		\lambda\bigl(\mathbf{X}^{n_0}_{i,\jj}(t),\varrho^{n_0}(t,\mathbf{X}^{n_0}_{i,\jj}(t))\bigr),
		\lambda\bigl(\mathbf{\bar X}_{\jj}(t),\varrho(t,\mathbf{\bar X}_{\jj}(t))\bigr)
	\right\}$
and
$\check\lambda^{n_0}(t)=$\linebreak $\max
\{		\lambda\bigl(\mathbf{X}^{n_0}_{i,\jj}(t),\varrho^{n_0}(t,\mathbf{X}^{n_0}_{i,\jj}(t))\bigr),
		\lambda\bigl(\mathbf{\bar X}_{\jj}(t),\varrho(t,\mathbf{\bar X}_{\jj}(t))\bigr)
\}$
and analogously for $\lambda_{\text{b}}$. 
By \eqref{eqn:Deltalambda} the area of $A_\jj$ is less than $T\,c_T\,\delta$ for any $\jj\in\JJ_{\bar N}$.
Therefore, taking $\delta$ sufficiently small
we have
\[
\P\left(\mathcal{N}_{1,\jj}(A_\jj)>0\text{ for any }\jj\in\JJ_{\bar N}\right)
<\eta.\]
\end{proof}

Now we are ready to prove Theorem \ref{thm:conv}. 
\begin{proof}[Proof of Theorem \ref{thm:conv}]
Fix $\epsilon>0$. We have to prove that 
$\P\left(\sup_{t\in[0,T]}d_{\mathbb{X}}\left(
\mathbbm{x}^{n_0}_1(t),\mathbbm{\bar x}(t)
\right)>\varepsilon \right)$
tends to 0 as ${n_0\to\infty}$.
To that end fix $\eta>0$, and take $R>0$ from Lemma \ref{lm:Omegaeta} for this $\eta$.
Take $\delta>0$ small enough for  Lemma \ref{lm:delta} to be satisfied and such that 
$c^{(1)}_{T,\eta}\delta\le\varepsilon $ in \eqref{eqn:Delta}. 
Since, by Theorem \ref{thm:convinD}, $\varrho^{n_0}$ 
converges to $\varrho$ 
in probability on $\mathcal{C}_T$ in the norm 
$\|f(\cdot ,\cdot )\|_{\mathcal{C}_T}=\sup_{t\in[0,T]}\|f(t,\cdot )\|_{C^1_b}$,
for sufficiently big $n_0$ we have
\eqref{eqn:supRrhoDrho} with probability $1-\eta $.
Using Lemma \ref{lm:delta} we know that with probability at least $1-3\eta$
we have $\mathbbm{\tilde x}^{n_0}_1=\mathbbm{x}^{n_0}_1$
and $\sup_{t\in[0,T]}d_{\mathbb{X}}\left(
\mathbbm{x}^{n_0}_1(t),\mathbbm{\bar x}(t)
\right)>\varepsilon $ by \eqref{eqn:Delta}. 
\end{proof} 


\begin{thebibliography}{10}
\expandafter\ifx\csname url\endcsname\relax
  \def\url#1{\texttt{#1}}\fi
\expandafter\ifx\csname urlprefix\endcsname\relax\def\urlprefix{URL }\fi
\expandafter\ifx\csname href\endcsname\relax
  \def\href#1#2{#2} \def\path#1{#1}\fi

\bibitem{Patlak1953}
C.~S. Patlak, Random walk with persistence and external bias, Bull. {M}ath.
  {B}iophys. 15 (1953) 311--338.
\newblock \href {http://dx.doi.org/10.1007/bf02476407}
  {\path{doi:10.1007/bf02476407}}.

\bibitem{KellerSegel1970}
E.~F. Keller, L.~A. Segel, Initiation of slime mold aggregation viewed as an
  instability, J. {T}heoret. {B}iol. 26~(3) (1970) 399--415.
\newblock \href {http://dx.doi.org/10.1016/0022-5193(70)90092-5}
  {\path{doi:10.1016/0022-5193(70)90092-5}}.

\bibitem{Horstmann2003}
D.~Horstmann, From 1970 until present: the {K}eller-{S}egel model in chemotaxis
  and its consequences. {I}, Jahresber. {D}eutsch. {M}ath.-{V}erein. 105~(3)
  (2003) 103--165.

\bibitem{HillenPainter2009}
T.~Hillen, K.~J. Painter, A user's guide to {PDE} models for chemotaxis, J.
  {M}ath. {B}iol. 58~(1-2) (2009) 183--217.
\newblock \href {http://dx.doi.org/10.1007/s00285-008-0201-3}
  {\path{doi:10.1007/s00285-008-0201-3}}.

\bibitem{Oelschlager84}
K.~Oelschl{\"a}ger, A martingale approach to the law of large numbers for
  weakly interacting stochastic processes., Ann. {P}robab. 12 (1984) 458--479.

\bibitem{Sznitman1984}
A.-S. Sznitman, Nonlinear reflecting diffusion process, and the propagation of
  chaos and fluctuations associated, J. {F}unct. {A}nal. 56~(3) (1984)
  311--336.
\newblock \href {http://dx.doi.org/10.1016/0022-1236(84)90080-6}
  {\path{doi:10.1016/0022-1236(84)90080-6}}.

\bibitem{Morale-et-al-05}
D.~Morale, V.~Capasso, K.~Oelschl{\"a}ger, An interacting particle system
  modelling aggregation behavior: from individuals to populations, J. {M}ath.
  {B}iol. 50 (2005) 49--66.

\bibitem{RudnickiWieczorek2006a}
R.~Rudnicki, R.~Wieczorek, Phytoplankton dynamics: from the behaviour of cells
  to a transport equation, Math. {M}odel. {N}at. {P}henom. 1~(1) (2006)
  83--100.

\bibitem{FournierJourdain2017}
N.~Fournier, B.~Jourdain, Stochastic particle approximation of the
  {K}eller-{S}egel equation and two-dimensional generalization of {B}essel
  processes, Ann. {A}ppl. {P}robab. 27~(5) (2017) 2807--2861.
\newblock \href {http://dx.doi.org/10.1214/16-{AAP}1267}
  {\path{doi:10.1214/16-{AAP}1267}}.

\bibitem{Wieczorek2015}
R.~Wieczorek, A stochastic particles model of fragmentation process with
  shattering, Electron. {J}. {P}robab. 20 (2015) no. 86, 1--17.
\newblock \href {http://dx.doi.org/10.1214/{EJP}.v20-4060}
  {\path{doi:10.1214/{EJP}.v20-4060}}.

\bibitem{BudhirajaFan2017}
A.~Budhiraja, W.-T. Fan, Uniform in time interacting particle approximations
  for nonlinear equations of {P}atlak-{K}eller-{S}egel type, Electron. {J}.
  {P}robab. 22 (2017) Paper No. 8, 37.
\newblock \href {http://dx.doi.org/10.1214/17-{EJP}25}
  {\path{doi:10.1214/17-{EJP}25}}.

\bibitem{CapassoWieczorek2020}
V.~Capasso, R.~Wieczorek, A hybrid stochastic model of retinal angiogenesis,
  Math. {M}ethods {A}ppl. {S}ci. 43~(18) (2020) 10578--10592.
\newblock \href {http://dx.doi.org/10.1002/mma.6725}
  {\path{doi:10.1002/mma.6725}}.

\bibitem{BoiCapassoMorale2000}
S.~Boi, V.~Capasso, D.~Morale, Modeling the aggregative behavior of ants of the
  species {P}olyergus rufescens, Vol.~1, 2000, pp. 163--176, spatial
  heterogeneity in ecological models ({A}lcal{\'a} de {H}enares, 1998).
\newblock \href {http://dx.doi.org/10.1016/{S}0362-546{X}(99)00399-5}
  {\path{doi:10.1016/{S}0362-546{X}(99)00399-5}}.

\bibitem{Stevens2000}
A.~Stevens, The derivation of chemotaxis equations as limit dynamics of
  moderately interacting stochastic many-particle systems, S{IAM} {J}. {A}ppl.
  {M}ath. 61~(1) (2000) 183--212.
\newblock \href {http://dx.doi.org/10.1137/{S}0036139998342065}
  {\path{doi:10.1137/{S}0036139998342065}}.

\bibitem{Oelschlager1989}
K.~Oelschl{\"a}ger, On the derivation of reaction-diffusion equations as limit
  dynamics of systems of moderately interacting stochastic processes, Probab.
  {T}heory {R}elated {F}ields 82~(4) (1989) 565--586.
\newblock \href {http://dx.doi.org/10.1007/{BF}00341284}
  {\path{doi:10.1007/{BF}00341284}}.

\bibitem{HwangKangStevens2005}
H.~J. Hwang, K.~Kang, A.~Stevens, Drift-diffusion limits of kinetic models for
  chemotaxis: a generalization, Discrete {C}ontin. {D}yn. {S}yst. {S}er. {B}
  5~(2) (2005) 319--334.
\newblock \href {http://dx.doi.org/10.3934/dcdsb.2005.5.319}
  {\path{doi:10.3934/dcdsb.2005.5.319}}.

\bibitem{BubbaLorenziMacfarlane2020}
F.~Bubba, T.~Lorenzi, F.~R. Macfarlane, From a discrete model of chemotaxis
  with volume-filling to a generalized {P}atlak-{K}eller-{S}egel model, Proc.
  {A}. 476~(2237) (2020) 20190871, 19.
\newblock \href {http://dx.doi.org/10.1098/rspa.2019.0871}
  {\path{doi:10.1098/rspa.2019.0871}}.

\bibitem{Meleard1996}
S.~M{\'e}l{\'e}ard, Asymptotic behaviour of some interacting particle systems;
  {M}c{K}ean-{V}lasov and {B}oltzmann models, in: Probabilistic models for
  nonlinear partial differential equations ({M}ontecatini {T}erme, 1995), Vol.
  1627 of Lecture {N}otes in {M}ath., Springer, {B}erlin, 1996, pp. 42--95.
\newblock \href {http://dx.doi.org/10.1007/{BF}b0093177}
  {\path{doi:10.1007/{BF}b0093177}}.

\bibitem{BolleyGuillinVillani2007}
F.~Bolley, A.~Guillin, C.~Villani, Quantitative concentration inequalities for
  empirical measures on non-compact spaces, Probab. {T}heory {R}elated {F}ields
  137~(3-4) (2007) 541--593.
\newblock \href {http://dx.doi.org/10.1007/s00440-006-0004-7}
  {\path{doi:10.1007/s00440-006-0004-7}}.

\bibitem{KotelenezKurtz2010}
P.~M. Kotelenez, T.~G. Kurtz, Macroscopic limits for stochastic partial
  differential equations of {M}c{K}ean-{V}lasov type, Probab. {T}heory
  {R}elated {F}ields 146~(1-2) (2010) 189--222.
\newblock \href {http://dx.doi.org/10.1007/s00440-008-0188-0}
  {\path{doi:10.1007/s00440-008-0188-0}}.

\bibitem{LiLiuYu2019}
L.~Li, J.-G. Liu, P.~Yu, On the mean field limit for {B}rownian particles with
  {C}oulomb interaction in 3{D}, J. {M}ath. {P}hys. 60~(11) (2019) 111501, 34.
\newblock \href {http://dx.doi.org/10.1063/1.5114854}
  {\path{doi:10.1063/1.5114854}}.

\bibitem{CapassoMorale2013}
V.~Capasso, D.~Morale, A multiscale approach leading to hybrid mathematical
  models for angiogenesis: the role of randomness, in: Mathematical methods and
  models in biomedicine, Lect. {N}otes {M}ath. {M}odel. {L}ife {S}ci.,
  Springer, {N}ew {Y}ork, 2013, pp. 87--115.

\bibitem{CapassoFlandoli2018}
V.~Capasso, F.~Flandoli, On the mean field approximation of a stochastic model
  of tumor-induced angiogenesis, European {J}ournal of {A}pplied {M}athematics
  30~(4) (2019) 619--658.
\newblock \href {http://dx.doi.org/10.1017/{S}0956792518000347}
  {\path{doi:10.1017/{S}0956792518000347}}.

\bibitem{WoodwardTysonMyerscoughMurrayBudreneBerg1995}
D.~Woodward, R.~Tyson, M.~Myerscough, J.~Murray, E.~Budrene, H.~Berg,
  Spatio-temporal patterns generated by {S}almonella typhimurium, Biophysical
  {J}ournal 68~(5) (1995) 2181--2189.
\newblock \href
  {http://dx.doi.org/https://doi.org/10.1016/{S}0006-3495(95)80400-5}
  {\path{doi:https://doi.org/10.1016/{S}0006-3495(95)80400-5}}.

\bibitem{Wang2000}
X.~Wang, Qualitative behavior of solutions of chemotactic diffusion systems:
  effects of motility and chemotaxis and dynamics, S{IAM} {J}. {M}ath. {A}nal.
  31~(3) (2000) 535--560.
\newblock \href {http://dx.doi.org/10.1137/{S}0036141098339897}
  {\path{doi:10.1137/{S}0036141098339897}}.

\bibitem{WangFordHarvey2008}
M.~Wang, R.~M. Ford, R.~W. Harvey, Coupled {E}ffect of {C}hemotaxis and
  {G}rowth on {M}icrobial {D}istributions in {O}rganic-{A}mended {A}quifer
  {S}ediments: {O}bservations from {L}aboratory and {F}ield {S}tudies,
  Environmental {S}cience \& {T}echnology 42~(10) (2008) 3556--3562.
\newblock \href {http://dx.doi.org/10.1021/es702392h}
  {\path{doi:10.1021/es702392h}}.

\bibitem{BansayeSimatos2015}
V.~Bansaye, F.~Simatos, On the scaling limits of {G}alton-{W}atson processes in
  varying environments, Electron. {J}. {P}robab. 20 (2015) no. 75, 36.
\newblock \href {http://dx.doi.org/10.1214/{EJP}.v20-3812}
  {\path{doi:10.1214/{EJP}.v20-3812}}.

\bibitem{CapassoMoraleFacchetti2012}
V.~Capasso, D.~Morale, G.~Facchetti, The role of stochasticity in a model of
  retinal angiogenesis, I{MA} {J}. {A}ppl. {M}ath. 77~(6) (2012) 729--747.

\bibitem{McDougallWatsonDevlinMitchellChaplain2012}
S.~R. McDougall, M.~G. Watson, A.~H. Devlin, C.~A. Mitchell, M.~A.~J. Chaplain,
  A hybrid discrete-continuum mathematical model of pattern prediction in the
  developing retinal vasculature, Bull. {M}ath. {B}iol. 74~(10) (2012)
  2272--2314.
\newblock \href {http://dx.doi.org/10.1007/s11538-012-9754-9}
  {\path{doi:10.1007/s11538-012-9754-9}}.

\bibitem{GarciaKurtz2006}
N.~L. Garcia, T.~G. Kurtz, Spatial birth and death processes as solutions of
  stochastic equations, A{LEA} {L}at. {A}m. {J}. {P}robab. {M}ath. {S}tat. 1
  (2006) 281--303.

\bibitem{BujorianuLygeros2006}
M.~L. Bujorianu, J.~Lygeros, Toward a general theory of stochastic hybrid
  systems, in: Stochastic hybrid systems, Vol. 337 of Lect. {N}otes {C}ontrol
  {I}nf. {S}ci., Springer, {B}erlin, 2006, pp. 3--30.
\newblock \href {http://dx.doi.org/10.1007/11587392_1}
  {\path{doi:10.1007/11587392_1}}.

\bibitem{BuckwarRiedler2011}
E.~Buckwar, M.~G. Riedler, An exact stochastic hybrid model of excitable
  membranes including spatio-temporal evolution, J. {M}ath. {B}iol. 63~(6)
  (2011) 1051--1093.

\bibitem{BonillaCapassoAlvaroCarretero2014}
L.~L. Bonilla, V.~Capasso, M.~Alvaro, M.~Carretero, Hybrid modeling of
  tumor-induced angiogenesis, Phys. {R}ev. {E} 90 (2014) 062716.
\newblock \href {http://dx.doi.org/10.1103/{P}hys{R}ev{E}.90.062716}
  {\path{doi:10.1103/{P}hys{R}ev{E}.90.062716}}.

\bibitem{Sznitman1991}
A.-S. Sznitman, Topics in propagation of chaos, in: {\'E}cole d'{\'E}t{\'e} de
  {P}robabilit{\'e}s de {S}aint-{F}lour {XIX}-1989, Vol. 1464 of Lecture
  {N}otes in {M}ath., Springer, {B}erlin, 1991, pp. 165--251.
\newblock \href {http://dx.doi.org/10.1007/{BF}b0085169}
  {\path{doi:10.1007/{BF}b0085169}}.

\bibitem{CattiauxGuillinMalrieu2008}
P.~Cattiaux, A.~Guillin, F.~Malrieu, Probabilistic approach for granular media
  equations in the non-uniformly convex case, Probab. {T}heory {R}elated
  {F}ields 140~(1-2) (2008) 19--40.
\newblock \href {http://dx.doi.org/10.1007/s00440-007-0056-3}
  {\path{doi:10.1007/s00440-007-0056-3}}.

\bibitem{Bauer2001}
H.~Bauer, Measure and integration theory, Vol.~26 of De {G}ruyter {S}tudies in
  {M}athematics, Walter de {G}ruyter \& {C}o., Berlin, 2001, translated from
  the {G}erman by {R}obert {B}. {B}urckel.
\newblock \href {http://dx.doi.org/10.1515/9783110866209}
  {\path{doi:10.1515/9783110866209}}.

\bibitem{KipnisLandim1999}
C.~Kipnis, C.~Landim, Scaling limits of interacting particle systems, Vol. 320
  of Grundlehren der {M}athematischen {W}issenschaften [{F}undamental
  {P}rinciples of {M}athematical {S}ciences], Springer-{V}erlag, {B}erlin,
  1999.
\newblock \href {http://dx.doi.org/10.1007/978-3-662-03752-2}
  {\path{doi:10.1007/978-3-662-03752-2}}.

\bibitem{JacodShiryaev1987}
J.~Jacod, A.~N. Shiryaev, Limit theorems for stochastic processes, Vol. 288 of
  Grundlehren der {M}athematischen {W}issenschaften [{F}undamental {P}rinciples
  of {M}athematical {S}ciences], Springer-{V}erlag, {B}erlin, 1987.
\newblock \href {http://dx.doi.org/10.1007/978-3-662-02514-7}
  {\path{doi:10.1007/978-3-662-02514-7}}.

\bibitem{Munkres2000}
J.~R. Munkres, Topology, Prentice {H}all, {I}nc., {U}pper {S}addle {R}iver,
  {NJ}, 2000, second edition of [ {MR}0464128].

\end{thebibliography}

\end{document}